\documentclass[11pt]{article}
\usepackage[T1]{fontenc}
\usepackage[ansinew]{inputenc}
\usepackage{amssymb,amsmath,amsthm,eucal,hyperref,mathrsfs,array,graphicx,xcolor,bbm,enumerate,dsfont}
\usepackage[all]{xy}

\usepackage{fullpage}

\numberwithin{equation}{section}

\newtheorem{lemma}{Lemma}[section]
\newtheorem{proposition}[lemma]{Proposition}
\newtheorem{theorem}[lemma]{Theorem}

\newtheorem{remark}[lemma]{Remark}

\DeclareMathOperator{\cov}{cov}

\renewenvironment{proof}{{\em Proof.}}{\hspace*{\fill} $\square$}
\newenvironment{proofof}[1]{{\em Proof of #1.}}{\hspace*{\fill} $\square$}

\title{\textsc{Critical Gaussian chaos: \\ convergence and uniqueness in the derivative normalisation} } 

\author{ Ellen Powell \thanks{ellen.g.powell@durham.ac.uk, supported by NCCR Swissmap and EPSRC grant EP/H023348/1} \\ \textit{Durham University} }
\date{}

 \newcommand{\R}{\mathbb{R}}
\newcommand{\C}{\mathbb{C}}

\newcommand{\I}{\mathds{1}}
\newcommand{\E}[1]{\mathbb{E}\left [ #1 \right ]}

\newcommand{\Os}{\mathcal{O}}
\newcommand{\F}{\mathcal F}
\newcommand{\e}{\operatorname{e}}

\newcommand{\eps}{\varepsilon}

\newcommand{\var}[1]{\text{var}(#1)}

\newcommand{\qbeh}[1]{\hat{\mathbb{Q}}^{\beta,\eps}\left[ #1 \right]}

\newcommand{\ebex}[1]{\mathbf{\tilde{Q}}_x^{\beta,\eps}\left[#1\right]}
\newcommand{\qbex}[1]{\mathbf{Q}_x^{\beta,\eps}\left[#1\right]}

\newcommand{\he}{h_\eps}
\newcommand{\tihe}{\tilde{h}_\eps}
\newcommand{\lex}{\lambda_\eps}

\begin{document}


\maketitle

\begin{abstract}
{We show that, for general convolution approximations to a large class of log-correlated fields, including the 2d Gaussian free field, the critical chaos measures with derivative normalisation converge to a limiting measure $\mu'$. This limiting measure does not depend on the choice of approximation. Moreover, it is equal to the measure obtained using the Seneta--Heyde renormalisation at criticality, or using a white-noise approximation to the field. }
\end{abstract}

\section{Introduction} The theory of Gaussian multiplicative chaos was developed by Kahane, \cite{kahane}, in order to rigorously define measures of the form
\[\mu^\gamma (dx):= \e^{\gamma h(x)-\frac{\gamma^2}{2}\mathbb{E}[h(x)^2]} \, dx\] 
where $h$ is a rough centered Gaussian field, satisfying certain assumptions, and $\gamma>0$ is a real parameter. Since $h$ is not defined pointwise, a regularisation procedure is required to define $\mu^\gamma$. In \cite{kahane}, it is assumed that the covariance kernel $K$ of $h$ is $\sigma$-positive, meaning that $K$ can be approximated by a series of smooth positive kernels $K_n$. It is then possible to associate to such an approximation the sequence of measures $\mu_n(dx):=\exp\{\gamma h_n(x)-(\gamma^2/2) \var{h_n(x)}\} dx$. Kahane proved that these measures converge as $n\to \infty$, and that the limit is independent of the choice of approximation. We call this limit the $\gamma$-chaos measure associated to $h$.

However, $\sigma$-positivity can be hard to check pointwise, and in recent years this theory has been significantly generalised by several authors \cite{RV,Ber,JS,Shamov}. When $K$ is not $\sigma$-positive, a natural way to approximate $h$ is to convolve it with a general mollifier function $\theta$. Writing $h_\eps$ for these regularisations, it has been shown that for log-correlated $h$, and under very general conditions on $\theta$, the approximate measures \begin{equation}
\label{eqn::mueps}
\mu_\eps^\gamma(dx):=\e^{\gamma h_\eps(x)}\e^{-\frac{\gamma^2}{2}\var{h_\eps(x)}} dx
\end{equation} converge weakly in law \cite{RV} and in probability \cite{Ber, Shamov} as $\eps \to 0$. The limit is non-zero if and only if $\gamma^2<2d$. Moreover, it is universal in that it does not depend on the choice of regularisation \cite{Ber, JS, Shamov}.

When $\gamma^2=2d$, an additional renormalisation is required in order to yield a non-trivial limiting measure. Motivated by the theory of multiplicative cascades and the branching random walk \cite{BK, AiSh} one can hope to renormalise at criticality in one of two different ways. The first is called the Seneta--Heyde renormalisation, and involves premultiplying the sequence of measures (\ref{eqn::mueps}) by the deterministic sequence $\sqrt{\log(1/\eps)}$. The other is a random renormalisation, which is defined by taking a derivative of the measure (\ref{eqn::mueps}) in $\gamma$. It has been shown in \cite{DSRV,DSRV2} that for a special class of fields $h$ having so-called $\star$-scale invariant kernels, and for a specific sequence of approximations to $h$, both procedures yield the same non-zero limiting measure (up to a constant). However, the result in these papers relies heavily on the cut-off approximation used for the kernel of $h$, and does not generalise to arbitrary convolution approximations. These are somewhat more natural, local approximations to the field, and the goal of the paper will be to extend the theory to this set-up.

In this paper we will be particularly, but not exclusively, interested in the specific case where the underlying field $h$ is a 2d Gaussian free field with zero-boundary conditions. In this case the measure $\mu^\gamma$ (when it is defined and non-zero) is known as the Liouville measure with parameter $\gamma$. This has been an object of considerable recent interest due to its strong connection with 2d Liouville quantum gravity and the KPZ relations \cite{DS, RhodesV,Ber2}. Recent works in the case $\gamma<2$ include \cite{DS, RhodesV, Ber}, which among other things make an in-depth study of its moments, multifractal structure, and universality. Recently, in \cite{uslqg}, it has also been shown that these measures can be approximated using so-called local sets of the Gaussian free field. This is a particularly natural construction because it is both local and conformally invariant. 

The critical case $\gamma=2$ has also been considered for the Gaussian free field: \cite{DSRV2,HRVdisk, JS, uslqg}. In \cite{DSRV2}, the authors generalised their construction for $\star$-scale invariant kernels to show convergence in the Seneta--Heyde and derivative renormalisations for a specific ``white noise'' approximation to the field. These both yield the same (up to a constant) non-trivial limiting measure $\mu'$, that we will call the critical Liouville measure. However, this proof again does not extend to convolution approximations. 

The purpose of this article is to complete the picture for convolution approximations to critical chaos. We will focus specifically on the case of the 2d GFF, and fields with $\star$-scale invariant kernels (in any number of space dimensions). This builds on recent work of Junnila and Saksman \cite{JS} (and also \cite{HRVdisk} in the case of the free field), who show that in either of the cases above, the critical measure can be constructed using convolution approximations in the Seneta--Heyde renormalisation. 

To complete the story, therefore, it remains to show that the random ``derivative'' renormalisation procedure will also yield the same limit for general convolution approximations. This is the main result of the current paper. We remark that the derivative renormalisation is somewhat more natural, and in fact, it is usually easier to show convergence of this before convergence in the Seneta--Heyde renormalisation (which is then obtained by a comparison argument). Here we will reverse this procedure. 

Suppose that $h$ is a log-correlated field in a bounded subset $D\subset \R^d$ with kernel $K(x,y)$. By this we mean that $(h,\rho)_{\rho\in \mathcal{M}}$ is a centered Gaussian process, indexed by the set of signed measures $\rho$ whose positive and negative parts $\rho^{\pm}$ satisfy $\iint \rho^{\pm}(dx)|K(x,y)|\rho^{\pm}(dy)<\infty$, with covariance structure 
\[ \cov((h,\rho)(h,\rho'))=\iint \rho(dx)K(x,y)\rho'(dy) \]
for $\rho,\rho'\in \mathcal{M}$. Also suppose that  $\theta$ is a positive measure such that
\begin{equation}\label{eqn::cond_theta} \theta\in \mathcal{M}, \, \text{supp}(\theta)\subseteq \overline{B(0,1)}, \, \int_{\overline{B(0,1)}} \theta(dx) =1 \text{ and } \int \frac{1}{{|u-v|}} \, \theta(du)=\text{O}(1)\end{equation}
uniformly over $v\in B(0,5)$. Then we define a sequence of $\theta$-mollified approximations to $h$ by setting for $\eps>0$, \begin{equation}
\label{eqn::conv_def}
h_\eps:=h\star \theta_{\eps}(x)=(h,\theta_{\eps,x}),\end{equation} where $\theta_{\eps}$ is the image of $\theta$ under the map $y\mapsto \eps y$ and $\theta_{\eps,x}$ is the image of $\theta_{\eps}$ under the map $y\mapsto y+x$. We define the measures $M_\eps$ and $D_\eps$ associated with this approximation by setting 
\begin{align*} M_\eps(\Os) & :=\int_{\Os} \e^{\sqrt{2d}h_{\eps}(x)-d\var{h_\eps(x)}} dx; \\
D_\eps(\Os) & :=\int_{\Os} (-h_\eps(x)+\sqrt{2d}\var{h_\eps(x)})  e^{\sqrt{2d}h_{\eps}(x)-d \var{h_\eps(x)}} dx\;\; (\Os\subset D).  \end{align*} Note that $M_\eps$ is exactly the same as $\mu_\eps^{\sqrt{2d}}$ (but we introduce the new notation to distinguish the special case $\gamma=\sqrt{2d}$ and avoid confusing notation when $d=2$). Our aim will be to prove the following.
\begin{theorem} \label{proposition::convergence}
	Suppose that $h$ is a 2d Gaussian free field in $D\subset \R^2$ bounded and simply connected. Suppose that $D_\eps$ is defined as above, for a mollifier $\theta$ satisfying (\ref{eqn::cond_theta}). Then $D_\eps$ converges weakly in probability as $\eps\to 0$ to the critical Liouville measure $\mu'$ constructed in \cite{DSRV2}. In particular $\lim_\eps D_\eps$ does not depend on $\theta$. 
\end{theorem}

\begin{theorem} \label{theorem::convergencestar}
	Suppose that $h$ is a Gaussian field in $\R^d$ ($d\ge 1$) with $\star$-scale invariant kernel and $D_\eps$ is defined as above, for a mollifier $\theta$ satisfying (\ref{eqn::cond_theta}) and with H\"{o}lder continuous density. Then $D_\eps$ converges weakly in probability as $\eps\to 0$ to a limiting measure. This measure is independent of the choice of approximation, and agrees with the critical measure constructed in \cite{DSRV,JS} (see Theorems \ref{theorem::critical_star} and \ref{theorem::critical_sh}).  
\end{theorem}

There is one further motivation for proving Theorem \ref{proposition::convergence}. In \cite{uslqg}, the authors construct a critical measure for the Gaussian free field, using a simple and natural approximation based on its local sets. This is closely related to the classical construction of multiplicative cascades \cite{kp}, and we believe that this connection can be exploited to help us improve our understanding of the situation at criticality (in particular, to prove a conjecture given in \cite{DSRV}.) However, it is a priori hard to connect the measure of \cite{uslqg} to the measure $\mu'$ of \cite{DSRV2}. It turns out that Theorem \ref{proposition::convergence} is exactly what is needed to show that they are in fact equal (for details of this argument, see \cite{uslqg}). In conclusion, Theorem \ref{proposition::convergence} gives us a universality statement for critical Liouville quantum gravity, that is now in line with the statement for the subcritical case \cite{Ber,Shamov,uslqg}.

\bigbreak
\noindent \textbf{Outline} We will begin in Section \ref{sec::prelims} by giving a brief introduction to log-correlated fields, and explaining how to approximate them using general mollifiers. We will also discuss here some of the existing literature concerning subcritical and critical Gaussian multiplicative chaos, and recall some basic facts about the 3-dimensional Bessel process. These occur naturally in critical Gaussian multiplicative chaos; roughly, as the value of the field locally about a typical point, and will be instrumental in the proof of Theorems \ref{proposition::convergence} and \ref{theorem::convergencestar}. In Section \ref{sec::freefield} we concentrate on the case when $h$ is a 2d Gaussian free field, and prove Theorem \ref{proposition::convergence}. We begin in Section \ref{sec::ui} by showing that certain families of ``cut-off'' approximations to the derivative measures (that we shall call $D_\eps^\beta$) are uniformly integrable. In fact, this will not be used directly in the proof of Theorem \ref{proposition::convergence}, but is needed for the aforementioned application to \cite{uslqg}, and introduces technical facts required for the rest of the proof. Section \ref{sec::convergence} contains the bulk of the proof. The main idea is to connect the derivative measures $D_\eps$ with the renormalised measures $\sqrt{\log(1/\eps)}M_\eps$, which we know converge by \cite{JS,HRVdisk}. To do this, we use a technique similar to that first applied in \cite{AiSh}, and then in \cite{DSRV2}, although the details of the proof are quite different. This is centred around the fact that for the circle average approximation to the free field, there is a natural ``rooted measure'' arising from the definition of $D_\eps$, under which it becomes a 3d Bessel process. We can also show that for a general convolution approximation, under the corresponding rooted measure, the process is approximately a Bessel (unfortunately, this introduces many technicalities in the proof.)  Properties of the Bessel process then allow us to conclude. Finally, in Section \ref{sec::starscale}, we show how the proof can be adapted for the case of $\star$-scale invariant kernels, to give Theorem \ref{theorem::convergencestar}. 

\bigbreak
\noindent \textbf{Acknowledgements} {I am especially grateful to Juhan Aru, Nathana\"{e}l Berestycki and Avelio Sep\'{u}lveda for many invaluable discussions concerning this paper, and Liouville measures in general. I would also particularly like to thank Nathana\"{e}l Berestycki for useful comments on a preliminary draft of the article. Further thanks are due to Wendelin Werner for inviting me to visit ETH, where the idea for this project originated, and to Vincent Vargas, for useful advice on general chaos measures. Finally, I am extremely grateful to Hubert Lacoin for pointing out an error in a previous version of the article.}

\section{Preliminaries}\label{sec::prelims}
\subsection{Log-correlated fields, 2d Gaussian free field and $\star$-scale invariant kernels.}
Let us recap the definition of log-correlated fields from the introduction. Suppose we have a non-negative definite kernel $K(x,y)$ on $D\subset \R^d$ of the form 
\begin{equation}
\label{eqn::kernelform}
K(x,y)=\log(|x-y|^{-1}) + g(x,y)\end{equation} where $g$ is a $C^1$ function on $\bar{D}\times \bar{D}.$ As in the introduction, we let $\mathcal{M}$ be the set of signed measures $\rho:=\rho^+-\rho^-$ whose positive and negative parts satisfy \[ \iint_{D\times D}|K(x,y)|\rho^{\pm}(dx)\rho^{\pm}(dy)<\infty.\]
The centered Gaussian field $h$, with covariance $K(x,y)$, is then defined as in \cite{Ber} to be the unique centred Gaussian process $(h,\rho)_{\rho\in \mathcal{M}}$ indexed by $\mathcal{M}$, such that
\[\cov((h,\rho),(h,\rho'))=\iint_{D\times D} K(x,y)\rho(dx)\rho'(dy)\]
for all $\rho,\rho'\in \mathcal{M}$.

We say that a kernel $K$ is $\star$-scale invariant if it takes the form
\begin{equation}\label{definition::starscale}K(x,y)=\int_1^\infty \frac{k(u|x-y|)}{u} \, du \end{equation} for 
$k:[0,\infty)\to \R$ a compactly supported and positive-definite $C^1$ function with $k(0)=1$. One can easily check that such a $K$ indeed has the form (\ref{eqn::kernelform}). Although this does not cover all kernels satisfying (\ref{eqn::kernelform}) it is still a natural family to consider, due to the nice scaling
relations it possesses, \cite{starscalekernels}. Moreover, the sequence of ``cut off'' approximations to $K$ given by
\[K_\eps(x,y)=\int_1^\frac{1}{\eps} \frac{k(u|x-y|)}{u} \, du\]
yields a family of approximating fields that exhibit a useful decorrelation property (see the proof of Theorem \ref{theorem::convergencestar}). 

As mentioned in the introduction, we will also be interested in the special case when $h$ is a 2-dimensional Gaussian free field. To define this, let $D\subset \C$ be a simply-connected domain. Then the zero boundary Gaussian free field $h$ on $D$ is defined as above, to be the log-correlated field whose kernel $K$ is given by the Green function, $G_D$, for the Laplacian on $D$. This satisfies 
\begin{equation}
\label{eqn::green}
G_D(x,y)=-\log|x-y|+g(x,y)\end{equation} for $g$ a smooth function on $\bar{D}\times \bar{D}$.

One feature that makes the Gaussian free field particularly nice to work with is that it satisfies the following spatial Markov property: if $A\subset D$ is a closed subset, then we can write $h=h^A+h_A$ where $h^A,h_A$ are independent, $h^A$ is a zero-boundary GFF on $D\setminus A$, and $h_A$ is harmonic when restricted to $D\setminus A$. We will see how this is useful to us in Section \ref{sec::freefield}.

In the following we will always assume, for technical reasons and without loss of generality, that our domain $D\subset \R^d$ contains the ball of radius $10$ around the origin. 

\subsection{Convolution with mollifiers}

Suppose we have a field $h$ with kernel $K$ satisfying (\ref{eqn::kernelform}). As discussed in the introduction, since $h$ is not defined pointwise, we need to use a regularisation procedure to define its chaos measures. A natural approach is to convolve $h$ with an approximation to the identity. Let $\theta$ be a non-negative Radon measure on $\R^d$, satisfying \eqref{eqn::cond_theta} and define the convolution approximations $(h_\eps(x))_{\eps>0}$ as in (\ref{eqn::conv_def}). The assumption (\ref{eqn::cond_theta}) on $\theta$ will be important to show various properties of the convolution approximations later on (cf. Lemma \ref{lemma::cov_h} and Corollary \ref{rmk::min_particle}). We remark here that (\ref{eqn::cond_theta})matches the condition given in \cite{Ber}, and includes most of the important examples. In particular, it includes the case when $\theta$ is uniform measure on the unit circle, or when $\theta$ has an $L^p$ density with respect to Lebesgue measure for some $p>2$.

We have the following estimate for the covariances of $(h_\eps)_\eps$: 

\begin{lemma}[\cite{Ber}]\label{lemma::cov_h} 
	Suppose $\theta$ satisfies (\ref{eqn::cond_theta}) and $h_\eps$ is defined as above. Then:
	\begin{equation}
	\label{eqn::cov_h}
	\cov(h_\eps(x),h_{\eps'}(y))=\log(1/(|x-y|\vee \eps \vee \eps'))+\text{\emph{O}}(1).
	\end{equation}
	where by $\text{\emph{O}}(1)$ we mean something that is uniformly bounded in $\eps,\eps'$, and $x,y$.
	
\end{lemma}

\begin{remark}Similarly, whenever we use order notation in the sequel, we will mean the order in $\eps$, uniformly in whatever spatial position(s) we are considering.
\end{remark}
\subsection{Maxima of the mollified fields} It will also be important for us in this article to get a hold of how fast our approximations $h_\eps$ can blow up. The Lemma below follows from classical Gaussian estimates, along the same lines as \cite[Proposition 19]{DSRV2} and \cite[Proposition 2.4] {Lacoin} (although it is actually a weaker bound than either of these).
%

\begin{lemma}\label{rmk::min_particle}
	Suppose that $\theta$ is a positive Radon measure on $\R^d$ satisfying (\ref{eqn::cond_theta}), and that $h$ has covariance kernel $K$ satisfying (\ref{eqn::kernelform}). Assume further that $\Os\subset \R^d$ is bounded. Then 
	\[ \inf_{\eps} \inf_{x\in \Os} \{-h_\eps(x)+\sqrt{2d}\log(1/\eps) \} > - \infty \]
	almost surely.
\end{lemma}

\begin{proof} 
	Without loss of generality we assume that $\mathcal{O}=[0,1]^d$. 
	Following the proof of \cite[Proposition 2.4]{Lacoin} we let, for $n\in \mathbb{N}$ and $\mathbf{i}\in I_n:=\{ \mathbf{i}\in \mathbb{Z}^d \,: \, \mathbf{i}e^{-n}\in[0,1]^d\}$, 
	\[ Y_{n,\mathbf{i}}:= h_{e^{-n}}(\mathbf{i}e^{-n}) \text{ and } Z_{n,\mathbf{i}}:=\textstyle\sup_{{\eps\in (e^{-(n+1)},e^{-n}],  x\in \mathbf{i}e^{-n}+[0,e^{-n})^d}} \, \big( h_\eps(x)-Y_{n,\mathbf{i}}\big). \]
	Since $Y_{n,\mathbf{i}}$ is centered and Gaussian with variance equal to $n+\mathrm{O}(1)$ uniformly in $\mathbf{i}$ (by Lemma \ref{lemma::cov_h}), the result of Lemma \ref{rmk::min_particle} follows from an easy application of the Borel-Cantelli lemma (exactly as carried out in the proof of \cite[Proposition 2.4]{Lacoin}), as soon as we can show that the $(Z_{n,\mathbf{i}})$ are uniformly sub-Gaussian. That is, it suffices to prove that for some $c>0$, we have
	\begin{equation}\label{eqn:unif_subgaussian}
		\mathbb{P}(Z_{n,\mathbf{i}}\ge \lambda) \le 2 \exp(-c\lambda^2))
	\end{equation}
for all $n\in \mathbb{N}, \mathbf{i}\in I_n$ and $\lambda \ge 0$.

For this, note that by Lemma \ref{lemma::cov_h} again, $\textstyle\sup_{{\eps\in (e^{-(n+1)},e^{-n}],  x\in \mathbf{i}e^{-n}+[0,e^{-n})^d}} \, \mathbb{E}[\big( h_\eps(x)-Y_{n,\mathbf{i}}\big)^2]$ is uniformly bounded in $\mathbf{i},n$. This means that by the Borrel-TIS inequality, \eqref{eqn:unif_subgaussian} holds as long as 
\[ \sup_n \sup_{i\in I_n} \mathbf{E}(Z_{n,\mathbf{i}})<\infty.\]
This in turn follows from Fernique's majorizing criterion, exactly as in \cite[Proposition 2.4]{Lacoin}, if we can prove that \[d_n((s,x),(t,y)):=\mathbb{E}\big[\big(h_{e^{-n}e^{-s}}(e^{-n}x)-h_{e^{-n}e^{-t}}(e^{-n}y)
\big)^2\big]^{1/2}\]
is uniformly H\"{o}lder continuous on $[0,1]\times[0,1]^d$.

To show this, we first write
	\begin{align*}& \mathbb{E}[(h_\eps(x)-h_\eps(y))^2] = \iint 2\log\left|\frac{x-y+\eps(v-w)}{\eps(v-w)}\right|\theta(dv)\theta(dw)\\
	& +\iint [g(x+\eps v,x+\eps w)+g(y+\eps v, y+\eps w)-g(x+\eps v, y+\eps w)-g(y+\eps v, x+\eps w)] \theta(dv)\theta(dw),\end{align*}
	which is bounded above by a universal constant times $(|x-y|/\eps)$ for $|x-y|\le \eps$, using that $g$ has bounded derivative on $[0,1]^d$, the assumption (\ref{eqn::cond_theta}) on $\theta$, and the  bound
	$\log(1+a)\le a \text{ for } a\ge 0.$
	By the triangle inequality, this implies that $d_n((s,x),(t,y))^2 \le |t-s|+|x-y|^{1/2}$ for all $(s,x)$ and $(t,y)$ in $[0,1]\times[0,1]^d$, which completes the proof.

%
\end{proof}

\subsection{Previous works on subcritical and critical Gaussian multiplicative chaos} 

As discussed in the introduction, Gaussian multiplicative chaos theory is a framework we can use to make sense of measures of the form ``$\e^{\gamma h(x)-\gamma^2/2\var{h(x)}}dx$'' for log-correlated Gaussian fields $h$. This stems from the classical martingale theory of the branching random walk \cite{Biggins, K} and multiplicative cascades \cite{kp}, and was initiated by Kahane \cite{kahane} in the 1980's. In the special case where $h$ is a 2d Gaussian free field, the Gaussian multiplicative chaos measure is often referred to as the Liouville measure \cite{DS}. Here we will state precisely some of the results mentioned in the introduction. 

When $\gamma<\sqrt{2d}$ (the subcritical regime) there are various approximation procedures that can be used to construct the chaos measure with parameter $\gamma$. One natural choice is to use the convolution approximations $h_\eps$ described in the previous section, and define approximate measures $\mu_\eps^\gamma$ by setting 
\begin{equation}
\label{eqn::mepsgamma}
\mu^\gamma_\eps(dx) := \mathbb{E}[\e^{\gamma h_\eps(x)}]^{-1}\e^{\gamma h_\eps(x)}dx
\end{equation}
for $\eps>0$. Note that the normalisation factor here is equal to $\eps^{\frac{\gamma^2}{2}}$ (up to a bounded constant that depends on $x$). We have the following result.

\begin{theorem}[\cite{Ber}]\label{theorem::subcriticalgmc}
	For $\gamma < \sqrt{2d}$ the measures $\mu^\gamma_\eps$ converge to a non-trivial measure $\mu^\gamma$ weakly in probability. Moreover, for any fixed Borel set $\Os$ we have that $\mu^\gamma_\eps(\Os)$ converges in $L^1$.
\end{theorem}

We emphasise that this limit $\mu^\gamma$ does not depend on the choice of mollifier $\theta$. In fact, one can approximate the field in other, completely different ways (for instance using a Karhunen--Lo\`{e}ve expansion \cite{Ber}) and find the same limit. For the case of the 2d free field, this will even work for ``non-Gaussian'' approximations. Indeed, in \cite{uslqg} the authors construct (the same) Liouville measure for $\gamma<2$ using sequences of so-called ``local sets'' of the field.

For general $h$, the subcritical measures $\mu^\gamma$ with $\gamma<\sqrt{2d}$ are almost surely atomless, and assign positive mass to any open set. On the other hand, as discussed in the introduction, it is known that for $\gamma \geq \sqrt{2d}$, the measures $\mu^\gamma_\eps$ converge to zero \cite{RV}. To define the critical (and supercritical) measures we must therefore make an additional renormalisation. These cases turn out to be much more tricky to deal with than the subcritical case, in part because the limiting measure will not possess any moments of order greater than or equal to $1$. Consequently a complete theory is still lacking, but some progress has been made (see \cite{rvreview} for a survey). Here and in the rest of this paper we will discuss the critical case $\gamma=\sqrt{2d}$. 

\subsubsection{Critical measures} Motivated by the corresponding constructions for multiplicative cascades, \cite{BK, AiSh}, we expect to be able to obtain a non-trivial measure at criticality using either of two renormalisation procedures: one deterministic and one random. Let us outline how this should work. Suppose you have some approximations $h_\eps$ to a log-correlated field $h$, that are continuous fields for each $\eps$. Then each of the following sequences should converge to the same (up to a constant) limiting measure.
\begin{itemize}
	\item The sequence of measures $\sqrt{\log(1/\eps)}\mu^{\sqrt{2d}}_\eps:=\sqrt{\log(1/\eps)}M_\eps$, where $\mu_\eps^\gamma$ is defined by (\ref{eqn::mepsgamma}). This is known as the Seneta--Heyde renormalisation.
	\item  The sequence of signed ``derivative'' measures, obtained by taking the derivative of $\mu_\eps^\gamma$ with respect to $\gamma$ and evaluating at $\gamma=\sqrt{2d}$. That is, the sequence 
	\[D_\eps(dz) := (-h_\eps(z)+\gamma \mathbb{E}[h_\eps(z)^2])\exp\left(\gamma h_\eps(z)- \frac{\gamma^2}{2}\E{h_\eps(z)^2} \right)dz\]
	(where we have also multiplied by $-1$ in order to yield a non-negative limit measure.)
\end{itemize} 

This statement was verified for a specific set-up in \cite{DSRV, DSRV2}.

\begin{theorem}[\cite{DSRV,DSRV2}]\label{theorem::critical_star}
	Suppose $h$ has a $\star$-scale invariant kernel $K(x,y)=\int_1^\infty k(u|x-y|)/u \, dx $ as in (\ref{definition::starscale}) and the approximate fields $h_\eps$ have kernels given by
	\[K_\eps(x,y):= \int_1^{1/\eps} \frac{k(u|x-y|)}{u} \, du.\] Then the two sequences of approximating measures described above converge weakly in probability to the same limiting measure, up to a constant $\sqrt{2/\pi}$. In particular, for any open set $\Os\subset \R^d$, $\sqrt{\pi/2}\sqrt{\ln(1/\eps)}M_\eps(\Os)$ and $D_\eps(\Os)$ converge in probability and in $L^p$ (any $p<1$) to the same limit. 
\end{theorem}
The authors in \cite{DSRV,DSRV2} were also able to generalise this approach to the case when $h$ is a 2d Gaussian free field, using a white-noise decomposition for the field and another specific sequence of ``cut-off'' approximations for the kernel. However, both of these proofs rely strongly on a martingale property satisfied by the choice of approximating fields $h_\eps$. In particular, they do not extend to general convolution approximations. 

Convolution is clearly a natural way to approximate the field $h$, and so we would like to have a version of Theorem \ref{theorem::critical_star} for such approximations. Using comparison techniques, Junnila and Saksman were able to do this for the Seneta--Heyde renormalisation.

\begin{theorem}[\cite{JS}]
	\label{theorem::critical_sh}
	Let $h$ be a $\star$-scale invariant field, and assume that in addition to (\ref{eqn::cond_theta}), the mollifier $\theta$ has a H\"{o}lder continuous density. Then the measures $\sqrt{\log(1/\eps)}M_\eps$ converge to a limiting measure weakly in probability as $\eps\to 0$. This limit is equal to $\sqrt{2/\pi}\mu'$ where \begin{itemize}
		\item $\mu'$ is the measure from Theorem \ref{theorem::critical_star},
		\item $\E{\mu'(\Os)}=\infty$ for any $\Os\subset \R^d$ and
		\item $\mu'(\Os)$ is positive almost surely for any $\Os\subset \R^d$.
	\end{itemize}
	Again we have that for any open set $\Os\subset \R^d$, $\sqrt{\pi/2}\sqrt{\ln(1/\eps)}M_\eps(\Os)$ and $D_\eps(\Os)$ converge in probability and in $L^p$ ($p<1$) to the same limit. 
\end{theorem}

This has also been proven for the 2d-Gaussian free field. 
\begin{theorem}[\cite{HRVdisk,JS}]\label{theorem::critical_sh_gff}
	Let $h$ be a 2d-GFF, and take any mollifier $\theta$ satisfying (\ref{eqn::cond_theta}). Then the measures $\sqrt{\log(1/\eps)}M_\eps$ converge to a limiting measure weakly in probability as $\eps\to 0$. This limit is equal to $\sqrt{2/\pi}\mu'$ where $\mu'$ is the critical Liouville measure of \cite{DSRV2}.
\end{theorem}

Note that Theorem \ref{theorem::critical_sh_gff} places a weaker constraint on the mollifier $\theta$. This is due to the proof given in \cite{HRVdisk}. The aim of this paper will be to prove the analogues of Theorems \ref{theorem::critical_sh} and \ref{theorem::critical_sh_gff} for the derivative renormalisation.

\subsection{Bessel processes}\label{sec::bessel}

To conclude this introduction, we need to recall some basic properties of Brownian motion; in particular, of the 3dimensional Bessel process. Let $\mathbf{P}$ denote the law of a standard Brownian motion $B_t$ in $\R$, started from a possibly random position $B_0$ such that $\mathbf{P}(B_0<0)>0$ and $B_0$ has finite exponential moments of all orders.\footnote{So that the expectation of  \eqref{eqn::martingale} lies in $(0,\infty)$.} Then it is easy to check that for any $\beta,\gamma>0$, the process \begin{equation}\label{eqn::martingale}(-B_t+\gamma\var{B_t}+\beta)\I_{\{-B_u+\gamma\var{B_t}+\beta>0\; \forall u\in[0,t]\}} \e^{\gamma B_t-\frac{\gamma^2}{2}\var{B_t}}\end{equation} 
is a non-negative martingale. Let $\mathcal{F}_t$ be the filtration generated by the Brownian motion and define a new measure $\mathbf{Q}$ by letting its Radon--Nikodym derivative when restricted to $\mathcal{F}_t$ be given by the martingale (normalised to have expectation one) at time $t$. One can check that this yields a well-defined law $\mathbf{Q}$, under which the process $(-B_t+\gamma\var{B_u}+\beta)_{t\geq 0}$ is a 3d Bessel process started from $-B_0+\gamma \var{B_0}+\beta$. Note that this starting position will also be biased, and will be positive almost surely under $\mathbf{Q}$. The next lemma records some properties of the 3d Bessel process that we will use in our proofs.

\begin{lemma} \label{lemma::bessel} Let $(X_t)_{t\geq 0}$ be a 3d Bessel process started from a random (positive) position $X_0$ with finite variance, and law $\mathbf{Q}$. Then 
	\begin{enumerate}
		\item $\mathbf{Q}[\frac{1}{X_t}]=\sqrt{
			\frac{2}{\pi t}} + \text{\emph{o}}(t^{-1/2})$ where the error term is less than $\frac{2}{\sqrt{t}}(\frac{\mathbf{Q}[X_0^2]}{t}+\frac{\mathbf{Q}[X_0]}{\sqrt{t}})$.
		\item $\mathbf{Q}[\frac{1}{X_t^2}]\leq 2/t$, uniformly in the starting position.
		\item $\mathbf{Q}[\frac{\sqrt{u}}{\log(2+u)^2} \leq X_u \leq  (1+\sqrt{u\log(1+u)}) \;\text{eventually}\,]=1$ 
		\item \[\mathbf{Q}\left[\frac{\sqrt{u}}{R\log(2+u)^2}\leq X_u \leq R(1+\sqrt{u \log(1+u)}) \;\forall u\geq 0\right] \to 1\] as $R\to \infty$, uniformly over $X_0$ with $\mathbf{Q}[X_0]\leq K$ for any $K$.  
		\item $\mathbf{Q}[\frac{1}{X_t}\I_{\{X_t\leq t^{1/4}\}}]\leq \frac{C}{2t}$, uniformly in the starting position, where $C$ is an absolute constant.	
	\end{enumerate}
	
\end{lemma}

\begin{proof} (1),(2) and (5) are straightforward to verify using direct calculation and scaling arguments. (3) is a classical result due to Motoo \cite{Motoo} and then (4) follows by continuity and Markov's inequality.  
\end{proof}

\section{Proof of Theorem \ref{proposition::convergence}} \label{sec::freefield}

In this section we will work to prove Theorem \ref{proposition::convergence}. Recall that this concerns the case when the underlying field $h$ is a 2d Gaussian free field in a domain $D\subset \R^2$. For this choice of field, there is a particular convolution approximation, when $\theta$ is uniform measure on the unit circle, that plays an important role. We call this the circle average process and distinguish it by writing $\tilde{h}_\eps:=h\star \theta_{\eps}$. The Markov property of the field allows us to deduce the following:
\begin{lemma}
	\label{lemma::circle_average_bm}
	For each $x\in D$ and $\delta<d(x,\partial D)$, $\{ \tilde{h}_{\e^{-u}}(x): u\geq \log(1/\delta) \}$
	is a Brownian motion started from $\tilde{h}_{\delta}(x)$. 
\end{lemma}

We will also need to compare $\tilde{h}_\eps$ with a general convolution approximation $h_\eps$. 

\begin{lemma} \label{lemma::compare_average}
	Let $\he$ and $\tihe$ be the mollified and circle averages of $h$ at a point $x$ with $d(x,\partial D)>\eps$. Then we can write 
	\begin{equation}\label{eqn::decomp} \he(x) = \lex(x) \tihe(x)+Y_\eps(x)\end{equation} where 
	$\lex(x)=1+\text{\emph{O}}(\log(1/\eps)^{-1})$ (uniformly in $x$) and $Y_\eps(x)$ is independent of $\tilde{h}_\eps(x)$, Gaussian, and has mean $0$ and variance $\text{\emph{O}}(1)$. 
\end{lemma}
\begin{proof}
	For this, we observe (by an easy calculation using (\ref{eqn::green})) that 
	\[ \cov(h_\eps(x),\tilde{h}_\eps(x))=\log(1/\eps)+\text{O}(1)\]
	for any $x\in D$ and $\eps<d(x,\partial D)$. Let $\lambda_\eps(x):=\cov(h_\eps(x),\tilde{h}_\eps(x))/\cov(\tilde{h}_\eps(x),\tilde{h}_\eps(x))$, so that by direct calculation $\cov(h_\eps-\lambda_\eps \tilde{h}_\eps,\tilde{h}_\eps)=0$. Then by Gaussianity, $\tilde{h}_\eps$ and $Y_\eps:=h_\eps-
	\lambda_\eps\tilde{h}_\eps$ are independent. Using Lemmas \ref{lemma::circle_average_bm} and \ref{lemma::cov_h}, we see that the variance of $Y_\eps$ is $\text{O}(1)$ and that $\lambda_\eps=1+\text{O}(\log(1/\eps)^{-1})$.
\end{proof}
\begin{remark}
	We will often drop the $x$ from $\lex(x)$ when it is clear from the context.\end{remark}

\begin{lemma}
	\label{rmk::covYyhx}
	$Y_\eps(x)$ also has bounded covariances with $Y_\eps(y)$ and $\tilde{h}_\eps(y)$ for any $x,y\in D$. Moreover, for $\delta\geq\eps$, we have \[-\rho_\delta^\eps(x)/2:=\cov(Y_\eps(x),\tilde{h}_\delta(x))=\text{\emph{O}}(1),\] uniformly in $\eps,\delta$ and $x$.  
\end{lemma}

\begin{proof}
	The first claims follow using direct calculation similar to the above. For the final claim note that $\mathbb{E}[Y_\eps(x)\tilde{h}_\delta(x)]=\mathbb{E}[h_\eps(x)\tilde{h}_\delta(x)]-\lex(x)\mathbb{E}[\tilde{h}_\eps(x)\tilde{h}_\delta(x)]$ where both expressions on the right-hand side are $\log(1/\delta)+\text{O}(1)$.
\end{proof}
\begin{remark}
	Lemma \ref{lemma::compare_average} implies that $\rho_\eps^\eps(x)=0$	for all $\eps,x$.
\end{remark}

Let us now move on to the proof of Theorem \ref{proposition::convergence}. By standard arguments, see \cite{Ber}, we need only prove that $D_\eps(\Os)\to \mu'(\Os)$ in probability for each fixed $\Os \subset D$. In fact, without loss of generality we may assume that $\Os:=B(0,1)$ is the unit disc. From now on we will work with this assumption.

\subsection{A uniformly integrable family.}\label{sec::ui}

We know from \cite{DSRV2} that if $\mu'$ is the critical Liouville measure, $\mu'(\Os)$ has infinite expectation for any $\Os\subset D$. Therefore, we cannot hope to have $L^1$-convergence or uniform integrability of $D_\eps(\Os)$. Since we prefer to work with uniformly integrable families, we instead consider a sequence of cut-off approximations $D_\eps^\beta$ to $D_\eps$. It will be very important to choose these cut-offs correctly, but for the right choice they \emph{will} be uniformly integrable (for each $\beta$) and moreover, will converge as $\eps\to 0$ (albeit in some slightly unusual sense, see Lemma \ref{lemma::conv_quotient}). Obtaining the desired convergence in Theorem \ref{proposition::convergence} then amounts to letting $\beta\to \infty$ and using Lemma \ref{lemma::minimum_particle} to see that $D_\eps^\beta$ is actually very close to $D_\eps$ for large enough $\beta$.

So, let us fix $\eps_0> 0$, such that $B(x,\eps)\subset D$ for every $\eps\leq \eps_0$ and $x\in \Os$. Then we define for $\beta>0$ and $\eps\in(0,\eps_0]$, the ``cut-off'' approximations
\begin{align*} M_\eps^\beta(\Os)&:=\int_{\Os} \e^{2h_{\eps}(x)-2\var{h_\eps(x)}} \I_{L_\eps(x)} \I_{\{-h_\eps(x)+2\var{h_\eps(x)}+\beta>1\}} \, dx; \; \text{and}\\
D^\beta_\eps(\Os) &:=\int_{\Os} (-h_\eps(x)+2\var{h_\eps(x)}+\beta)  e^{2h_{\eps}(x)-2 \var{h_\eps(x)}} \I_{L_\eps(x)} \I_{\{-h_\eps(x)+2\var{h_\eps(x)}+\beta>1\}} ;  \end{align*}
where \[L_\eps(x):=\{-\tilde{h}_\delta(x)+2\lambda_\eps(x) \var{\tilde{h}_\delta(x)} + \beta -\rho_\delta^\eps(x) >0 ; \, \forall \, \delta\in [\eps,\eps_0] \}.\] Note that both $M_\eps^\beta(\Os)$ and $D_\eps^\beta(\Os)$ are positive by definition, and also that $M_\eps^\beta(\Os)$ $\leq D_\eps^\beta(\Os)$. For ease of notation we set
\begin{align*}
& f_{\eps,\gamma}^\beta(x) = -h_\eps(x)+\gamma \var{h_\eps(x)}+\beta
& g_{\eps,\gamma}(x) = \gamma\he(x)-(\gamma^2/2) \var{\he(x)} \\  & \tilde{f}_{\eps, \gamma}^{\beta}(x)  =-\tilde{h}_\eps(x)+\gamma\var{\tilde{h}_\eps(x)}+\beta 
& \tilde{g}_{\eps,\gamma}(x) = \gamma\tihe(x)-(\gamma^2/2)\var{\tihe(x)}\\
& f_{\eps,\gamma}^{\beta,Y}(x) = -Y_\eps(x) + \gamma \var{Y_\eps(x)}+\beta 
& g_{\eps,\gamma}^Y(x)  =\gamma Y_\eps(x)-(\gamma^2/2) \var{Y_\eps(x)}
\end{align*}
recalling the definition of $Y$ from Lemma \ref{lemma::compare_average}. Then we have 
\begin{align*} & M_\eps^\beta(\Os):=\int_{\Os} \e^{g_{\eps,2}(x)} \I_{L_\eps(x)} \I_{\{f_{\eps,2}^\beta(x)>1\}} \, dx ;\\ 
& D^\beta_\eps(\Os) :=\int_{\Os} f_{\eps,2}^\beta(x)  e^{g_{\eps,2}(x)} \I_{L_\eps(x)} \I_{\{f_{\eps,2}(x)^\beta>1\}} \, dx \end{align*} and 
\[ L_\eps(x)= \{\tilde{f}_{\delta,2\lex(x)}^{\beta}(x)-\rho_\delta^\eps >0 \, \forall \delta\in [\eps,\eps_0]\}.\]
The decomposition 
\begin{equation}\label{eqn::fdecomp} f_{\eps,2}^\beta(\cdot)=\lex(\cdot)\tilde{f}_{\eps,2\lex(\cdot)}^\beta(\cdot)+f^{0,Y}_{\eps,2}(\cdot)+(1-\lex(\cdot))\beta \;\; \text{and} \;\; g_{\eps,2}(\cdot)=\tilde{g}_{\eps,2\lex(\cdot)}(\cdot)+g^Y_{\eps,2}(\cdot).\end{equation} will also come in very useful in what follows.

\begin{proposition}\label{lem::ui}
	For fixed $\beta>0$, $(D_\eps^\beta(\Os))_{\eps\leq \eps_0}$ is a uniformly integrable family.  
\end{proposition}

\begin{proof}
	The proof of this Lemma is inspired by that of Berestycki \cite{Ber}, who shows uniform integrability of $\mu_\eps^\gamma$ in the subcritical case. In analogy to his approach, for $a\ge\eps>0$ we define the \emph{good event} $G_{\eps,a}^R(x):=$ \[ \left\{\frac{\sqrt{\log(1/u)}}{R\log(2+\log(1/u))^2} \leq \tilde{f}_{u,2\lambda_\eps(x)}^{\beta}(x) \leq  R(1+\sqrt{\log(1/u)\log(1+\log(1/u))}) \;\; \forall u\in [\eps, a]\right\}\] 
	and write $D_\eps^\beta(\Os)=J_\eps^\beta + \hat{J}_\eps^\beta$, where $J_\eps^\beta$ is the integral over all ``good'' $x$, for which $G_{\eps,\eps_0}^R(x)$ holds. \footnote{Note that we are setting $a=\eps_0$ here, but we define the more general notation $G_{\eps,a}^R$ for use later on.} The rationale behind choosing $G$ in this way is that it separates bad points of the field, which are ``too thick" and make the second moment explode, from the good points. 
	
	To conclude, it is enough to prove the following two lemmas.
	
	\begin{lemma}\label{lem::uipart1}
		$\mathbb{E}[\hat{J}_\eps^\beta]\leq p(R)$ for all $ \eps\leq \eps_0$ where $p(R)\to 0$ as $R\to \infty;$
	\end{lemma}
	\begin{lemma}\label{lem::uipart2}
		For fixed $R, J_\eps^\beta \, \text{is uniformly bounded in} \, L^2.$
	\end{lemma}
	We first give a very rough idea of why these should hold:
	\begin{itemize}
		\item $\mathbb{E}[\hat{J}_\eps^\beta]$ corresponds to the probability of $G^R_{\eps,\eps_0}(x)$ \emph{not} holding under a weighted law: specifically, under the law with Radon--Nikodym derivative (with respect to $\mathbb{P}$) proportional to 
		\[f_{\eps,2}^\beta(x)\e^{g_{\eps,2}(x)}\I_{L_\eps(x)}\I_{\{f_{\eps,2}^\beta(x)>1\}}.\]
		Under this law we know that $\tilde{f}^\beta_{u,2\lambda_\eps(x)}(x)$ is (approximately) a Bessel process. Thus we know by Lemma \ref{lemma::bessel} that this probability tends to $0$ as $R\to \infty$. 
		\item Now we move on to the $L^2$ bound. Every time we write $\approx$ it requires a lot of justification, usually because $h_\eps$ is not exactly a Brownian motion. First note that by the Markov property of the field and the fact that (\ref{eqn::martingale}) is a martingale, \begin{align*} & \mathbb{E}[f_{\eps,2}^\beta(x)f_{\eps,2}^\beta(y)\e^{g_{\eps,2}(x)}\e^{g_{\eps,2}(y)}\I_{L_\eps(x)}\I_{L_\eps(y)}] \\ & \approx \mathbb{E}[f_{\delta,2}^\beta(x)f_{\delta,2}^\beta(y)\e^{g_{\delta,2}(x)}\e^{g_{\delta,2}(y)}\I_{L_\delta(x)}\I_{L_\delta(y)}] \end{align*}
		for $x,y\in \Os$, where $\delta=\delta(x,y)=(|x-y|/3)\vee \eps$. Now, on the event $G_{\eps,\eps_0}^R(x)\cap G_{\eps,\eps_0}^R(y)$, \[f_{\delta,2}^\beta(x) \approx \sqrt{\log(1/\delta)}, \;f_{\delta,2}^\beta(y) \approx \sqrt{\log(1/\delta)} \; \text{ and} \; g_{\delta,2}(y)\approx-2\sqrt{\log(1/\delta)}+2\log(1/\delta).\] We can use this to show that, roughly, 
		\begin{align*} \mathbb{E}[f_{\delta,2}^\beta(x)f_{\delta,2}^\beta(y)\e^{g_{\delta,2}(x)}\e^{g_{\delta,2}(y)}\I_{G_{\eps,\eps_0}^R(x)}\I_{G_{\eps,\eps_0}^R(y)}] &\lesssim  \delta^{-2}\log (1/\delta) \e^{-2\sqrt{\log(1/\delta)}} \mathbb{E}[\e^{g_{\delta,2}(x)}] \\ &= \delta^{-2}\log (1/\delta) \e^{-2\sqrt{\log(1/\delta)}}. \end{align*}
		We then only need to verify that this function of $\delta(x,y)$ is integrable over $\Os\times \Os$. 
	\end{itemize}
	We prove Lemmas \ref{lem::uipart1} and \ref{lem::uipart2} below. As already mentioned, there are several technical difficulties with making the above argument rigorous.

\end{proof}

\begin{proofof}{Lemma \ref{lem::uipart1}} Consider for $x\in \Os$
	\begin{equation}\label{eqn::exp_good_events} \E{f_{\eps,2}^\beta(x)e^{g_{\eps,2}(x)} \I_{L_\eps(x)}\I_{\{f_{\eps,2}^\beta(x)>1\}} \I_{G_{\eps, \eps_0}^R(x)^c}}.
	\end{equation}
	To prove the lemma, we need to show that this converges to $0$ as $R\to \infty$, uniformly in $\eps$ and $x$. The strategy is to rewrite it as an expectation with respect to a different measure, under which we understand well the behaviour of $\tilde{f}_{\eps,2\lambda_\eps}(x)$. We set
	\[ \frac{d\tilde{\mathbf{Q}}_x^{\beta,\eps}}{d \mathbb{P}} = (\tilde{Z}_
	\eps^\beta(x))^{-1} \tilde{f}_{\eps,2\lambda_\eps}^\beta(x)\e^{g_{\eps,2}(x)} \I_{L_\eps(x)}; \;\;\; \tilde{Z}_\eps^\beta(x) = \E{\tilde{f}_{\eps,2\lambda_\eps}^\beta(x)\e^{g_{\eps,2}(x)} \I_{L_\eps(x)}}.\] This measure will be extremely important throughout the paper because, under $\tilde{\mathbf{Q}}_x^{\beta,\eps}$, the process 
	\[\{\tilde{f}_{u,2\lambda_\eps}^\beta(x)-\rho_u^\eps(x); \,u\in [\eps,\eps_0]\}\] is a time changed 3d Bessel process. To see why this is true, we split the weighting that defines $\tilde{\mathbf{Q}}_x^{\beta,\eps}$ into two steps. By decomposition (\ref{eqn::fdecomp}) we have $g_{\eps,2}(x)=\tilde{g}_{\eps,2\lex}(x)+g_{\eps,2}^Y(x)$ and so we can first consider what happens if we only weight by $\exp(g_{\eps,2}^Y(x))$. Let us call this intermediate law $\hat{\mathbb{P}}$. By the Cameron--Martin--Girsanov theorem, and the definition $\rho_\delta^\eps(x):=-2\cov(\tilde{h}_\delta(x),Y_\eps(x))$, the process \[-\tilde{h}_\delta(x)-\rho_{\delta}^\eps(x)\] is a time changed Brownian motion under $\hat{\mathbb{P}}$.  For the second step in the weighting we use the definition of $L_\eps(x)$, and the fact that $\rho_\eps^\eps(x)=0$. This means that this second step is simply the Bessel process weighting described in Section \ref{sec::bessel}, with $\gamma=2\lex(x)$. The same argument also implies that $\tilde{Z}_\eps^\beta(x)$ does not depend on $\eps$ for each $x$, since (\ref{eqn::martingale}) is a martingale.
	
	To prove the lemma, and we will apply this technique over and over again, we rewrite  (\ref{eqn::exp_good_events}) as  
	\[ \tilde{Z}_\eps^\beta(x) \mathbf{\tilde{Q}}_x^{\beta,\eps} \left[ \frac{f_{\eps,2}^\beta(x)}{\tilde{f}_{\eps,2\lambda_\eps}^\beta(x)}\I_{\{f_{\eps,2}^\beta>1 \}} \I_{G^R_{\eps, \eps_0}(x)^c}
	\right] \] where by Lemma \ref{lemma::bessel} part (4) we know that $\tilde{\mathbf{Q}}_x^{\beta,\eps}(G_{\eps,\eps_0}^R(x)^c)\to 0$ as $R\to \infty$, uniformly in $\eps$ and $x$ (using the uniform boundedness of $(\rho_\delta^\eps(x))_{\delta>\eps}$.)
	
	Using the fact that \[f_{\eps,2}^\beta(x)=\lex(x)\tilde{f}_{\eps,2\lex(x)}^\beta(x)+\text{O}(1)-Y_\eps(x),\] (cf. decomposition (\ref{eqn::fdecomp})), and H\"{o}lder's inequality, it is  enough for us to show that
	\[\mathbf{\tilde{Q}}_x^{\beta,\eps}\left[\left(\frac{|Y_\eps(x)|+\text{O}(1)}{\tilde{f}_{\eps,2\lex}^\beta(x)}\right)^{3/2}\right]=\text{O}(1).\] However, this follows by Cauchy--Schwarz, since
	\[\mathbf{\tilde{Q}}_x^{\beta,\eps}\left[\left(\frac{|Y_\eps(x)|+\text{O}(1)}{\tilde{f}_{\eps,2\lex}^\beta(x)}\right)^{3/2}\right]^2\leq \mathbf{\tilde{Q}}_x^{\beta,\eps}\left[\frac{(|Y_\eps(x)|+\text{O}(1))^3}{\tilde{f}^\beta_{\eps,2\lex}(x)} \right]  \mathbf{\tilde{Q}}_x^{\beta,\eps}\left[\frac{1}{\tilde{f}^\beta_{\eps,2\lex}(x)^2} \right] \] and
	\begin{itemize} \item $\mathbf{\tilde{Q}}_x^{\beta,\eps}[(\tilde{f}_{\eps,2\lex}^\beta(x))^{-2}]$ is bounded by Lemma \ref{lemma::bessel}, part (2); 
		\item 
		$\mathbf{\tilde{Q}}_x^{\beta,\eps}[|Y_\eps(x)|^p/\tilde{f}_{\eps,2\lambda_\eps}^\beta(x)]\lesssim \mathbb{E}[|Y_\eps(x)|^p\e^{g_{\eps,2}(x)}]$ is bounded for $p=1,2,3$.
	\end{itemize}
\end{proofof}
\medbreak

\begin{proofof}{Lemma \ref{lem::uipart2}} First, we make the simple bound 
	\begin{equation} \label{eqn::je2} \mathbb{E}[(J_\eps^\beta)^2]\leq \iint_{\Os^2} \E{
		|f_{\eps,2}^\beta(x)||f_{\eps,2}^\beta(y)|\e^{g_{\eps,2}(x)}\e^{g_{\eps,2}(y)}\I_{L_\eps(x)}\I_{L_\eps(y)}\I_{G^R_{\eps,\eps_0}(x)}\I_{G^R_{\eps,\eps_0}(y)}} \, dy \, dx  \end{equation}
	and fix some $x\in \Os$. For this fixed $x$, we will break the integral over $y$ into two parts: those with $|x-y|>3\eps$, and those with $|x-y|\leq 3\eps$. Let us begin with the first case. For such a $y$, we set $\delta=\delta(x,y):=|x-y|/3$, so that the $\delta$-balls around $x$ and $y$ are disjoint. We are going to use the fact that the circle averages around $x$ and $y$ decorrelate after this time. More precisely, if we let $\mathcal{H}$ be the $\sigma$-algebra generated by $h|_{D\setminus (B(x,\delta)\cup B(y,\delta))}$, then we have the following observations, which we state as a lemma.
	\begin{lemma}
		\label{lem::uiaugpart2} 
		\begin{enumerate}[(1)]
			\item 
			Conditionally on $\mathcal{H}$, the processes $(\tilde{h}_{\delta\e^{-t}}(x)-\tilde{h}_\delta(x))_{t\geq 0}$ and $(\tilde{h}_{\delta \e^{-t}}(y)-\tilde{h}_\delta(y))_{t\geq 0})$ are independent Brownian motions. 
			\item We can write $Y_\eps(x)=Y_\eps^1(x)+Y_\eps^2(x)$ and $Y_\eps(y)=Y_\eps^1(y)+Y_\eps^2(y)$  where:
			\begin{itemize}
				\item $Y_\eps^1(x)$ and $Y_\eps^1(y)$ are measurable with respect to $\mathcal{H}$;
				\item $Y_\eps^2(x)$ is independent of $\mathcal{H}$, $Y_\eps^2(y)$ and $(\tilde{h}_\eta(y)-\tilde{h}_\delta(y))_{\eta\leq \delta}$; 
				\item $Y_\eps^2(y)$ is independent of $\mathcal{H}$, $Y_\eps^2(x)$ and $(\tilde{h}_\eta(x)-\tilde{h}_\delta(x))_{\eta\leq \delta}$;
				\item $Y_\eps^i(x),Y_\eps^i(y)$ for $i=1,2$ have bounded variance; and 
				\item $2\cov(Y_\eps^2(x),\tilde{h}_\eta(x)-\tilde{h}_\delta(x))=-\rho^\eps_{\eta}(x)+\rho^\eps_{\delta}(x)$ (similarly if $x$ is replaced with $y$).
			\end{itemize} 
			\item We have \begin{align*}\mathbb{E}[\tilde{f}_{\eps,2\lambda_\eps}^\beta(x)\e^{g_{\eps,2}(x)}\I_{L_\eps(x)}\mid \mathcal{H}]  =(&\tilde{f}_{\delta,2\lambda_\eps}^\beta -\rho_\delta^\eps(x))\I_{\{\tilde f_{\eta,2\lambda_\eps} -\rho_\eta^\eps(x) >0 ; \, \forall \, \eta\in [\delta,\eps_0]\}}\\  \times &\e^{\tilde{g}_{\delta,2\lex}(x)}\e^{2Y^1_\eps(x)-2\text{\emph{var}}(Y^1_\eps(x))}\e^{2\rho_\delta^\eps(x)}.\end{align*}
			\item We also have
			$\mathbb{E}[\,|Y_\eps^2(x)|\e^{g_{\eps,2}(x)}\I_{L_\eps(x)}\mid \mathcal{H}]\leq C\e^{\tilde{g}_{\delta,2\lex}(x)}\e^{2Y^1_\eps(x)}$ where $C$ is a universal constant. 
			\item Items (3) and (4) also hold when $x$ is replaced by $y$.
		\end{enumerate}
	\end{lemma}
	\begin{proofof}{Lemma \ref{lem::uiaugpart2}}
		By the Markov property of the Gaussian free field, conditionally on $\mathcal{H}$ we can write $h=h^{\mathcal{H}}+h_{\mathcal{H}}$ where: 
		\begin{itemize}
			\item	
			$h_{\mathcal{H}}$ is measurable with respect to $\mathcal{H}$ and harmonic when restricted to $B(x,\delta)\cup B(y,\delta)$; and
			\item $h^{\mathcal{H}}$, independent of $\mathcal{H}$, is a sum of two independent zero boundary GFFs: one in $B(x,\delta)$ and one in $B(y,\delta)$.
		\end{itemize} We use this to prove the points in turn.
		\medbreak
		
		(1) This follows from the fact that $h_\mathcal{H}(x)=\tilde{h}_\delta(x)$ and $h_\mathcal{H}(y)=\tilde{h}_\delta(y)$ (by harmonicity), and the fact that the circle average process of a Gaussian free field is a Brownian motion.
		\medbreak
		
		(2) We have $Y_\eps(x)=h_\eps(x)-\lex(x)\tilde{h}_\eps(x)$ by definition and so we can write $Y_\eps^1(x)=(h_{\mathcal{H}},\theta_{\eps,x})-\lex(x) (h_{\mathcal{H}},\tilde{\theta}_{\eps,x})$ and $Y_\eps^2=(h^{\mathcal{H}},\theta_{\eps,x})-\lex(x) (h^{\mathcal{H}},\tilde{\theta}_{\eps,x})$, where $\tilde{\theta}$ is uniform measure on the unit circle. The claimed properties of this decomposition are easy to see.
		\medbreak
		
		(3)	We first take out the $\mathcal{H}$-measurable parts from the conditional expectation on the left-hand side. To this end we write for $\eta<\delta$
		\begin{align*} \tilde{f}_{\eta,2\lex}^\beta(x) &= \tilde{f}_{\delta,2\lex}^\beta(x)-\rho_\delta^\eps - (\tilde{h}_\eta(x)-\tilde{h}_\delta(x)) +2\lambda_\eps\log(\delta/\eta) +\rho_{\delta}^\eps\\
		&	:= W - (\tilde{h}_\eta(x)-\tilde{h}_\delta(x)) +2\lambda_\eps\log(\delta/\eta) +\rho_{\delta}^\eps\end{align*}
		where $W=\tilde{f}^\beta_{\delta,2\lambda_\eps}(x)-\rho_{\delta}^\eps$ is $\mathcal{H}$ measurable. We can also write 
		\begin{align*}& L_\eps(x) = L_\eps^1 \cap L_\eps^2 \\ & :=  \{\tilde{f}_{\eta,2\lex}^\beta(x)-\rho_\eta^\eps>0 \, \forall \, \eta\in [\delta,\eps_0]\} \cap \{-(\tilde{h}_\eta(x)-\tilde{h}_\delta(x))-(\rho_\eta^\eps-\rho_\delta^\eps) + W>0\, \forall \, \eta\in [\eps,\delta]\} 	
		\end{align*}
		
		where $L_\eps^1$ is also $\mathcal{H}$-measurable. Putting these together, using that $\rho_{\eps}^\eps=0$ and breaking up $g_{\eps,2}(x)$ using (\ref{eqn::fdecomp}) and point (2), we see that 
		\begin{align*} \mathbb{E}[&\tilde{f}_{\eps,2\lambda_\eps}^\beta(x)\e^{g_{\eps,2}(x)}\I_{L_\eps(x)}\mid \mathcal{H}] =\e^{2\lambda_\eps\rho_\delta^\eps(x)}\e^{\tilde{g}_{\delta,2\lex}(x)}\e^{2Y_\eps^1(x)-2\var{Y_\eps^1(x)}} \I_{L_\eps^1(x)} \;\; \times \\
		  \mathbb{E}[& (W-(\tilde{h}_\eps(x)-\tilde{h}_\delta(x))+2\lambda_\eps\log(\delta/\eps)+\rho_\delta^\eps)\e^{2\lambda_\eps(\tilde{h}_\eps(x)-\tilde{h}_\delta(x))-2\lambda_\eps^2\log(\delta/\eps)-2\lambda_\eps\rho_\delta^\eps(x)} \\ 
		 & \e^{2Y_\eps^2(x)-2\var{Y_\eps^2(x)}} \I_{L_\eps^2(x)}\mid \mathcal{H}]\end{align*}
		Now we can use Girsanov's theorem, as in the proof of Lemma \ref{lem::uipart1}, to get rid of the $\exp\{2Y_\eps^2(x)-2\var{Y_\eps^2(x)}\}$ term. More precisely, changing measure by $\exp\{2Y_\eps^2(x)-2\var{Y_\eps^2(x)}\}$ has the effect of shifting the law of $(\tilde{h}_\eta(x)-\tilde{h}_\delta(x))_{\eta\in[\eps,\delta]}$ by adding on the deterministic function $\rho_\delta^\eps(x)-\rho_\eta^\eps(x)$. We then see that the conditional expectation above is nothing but the expectation of the Brownian motion martingale (\ref{eqn::martingale}), starting from $W$. The result follows. 
		\medbreak
		
		(4) For this we bound the indicator $\I_{L_\eps(x)}$ by 1, and take out the parts which are measurable with respect to $\mathcal{H}$ as in part (3). Then we are left with the expectation of $|Y_\eps^2(x)|$ under a shifted law, where $Y_\eps^2(x)$ is still a Gaussian with $\text{O}(1)$ mean and variance (since $Y_\eps^2(x)$ has bounded covariances with everything.) This proves the claim.

	\end{proofof}
	\medbreak
	
	This lemma allows us to deduce that the integrand of (\ref{eqn::je2}), in the case $|x-y|>3\eps$, is less than or equal to some constant, depending on $\beta$ only, times 
	\begin{align*}\label{eqn::bound_L2} \mathbb{E}\big[&(\tilde{f}_{\delta,2\lambda_\eps(x)}^\beta(x)+1+|Y_\eps^1(x)|)
	(\tilde{f}_{\delta,2\lambda_\eps(y)}^\beta(y)+1+|Y_\eps^1(y)|)  \\
	& \times \e^{\tilde g_{\delta,2\lambda_\eps(x)}(x)}\e^{\tilde g_{\delta,2\lambda_\eps(y)}(y)}
	\e^{2Y_\eps^1(x)}\e^{2Y_\eps^1(y)}\I_{G^R_{\delta,\delta}(x)}\I_{G^R_{\delta,\delta}(y)}
	\big]
	\end{align*}
	where $\delta=\delta(x,y)=|x-y|/3$. Here we used that $\rho_\delta^\eps(\cdot)$ and $\var{Y_\eps^1(\cdot)}$ are uniformly bounded, and changed $G_{\eps,\eps_0}^R(\cdot)$ to the larger $\mathcal{H}$-measurable event $G_{\delta,\delta}^R(\cdot)$, so that it would not interfere with the conditioning step.
	
	Now we can use the definition of $G^R_{\delta,\delta}$. This, together with the fact that $Y_\eps^1(\cdot)$ has bounded variance and covariance with everything, tells us that the above is bounded by a constant times 
	\[ \delta^{-2} F_{R}(\log(1/\delta)) \; ; \;\; F_{R}(z):= R^2(1+\sqrt{z\log(1+z)})^2 \e^{-\frac{2\sqrt{z}}{{R}(\log(2+z))}}. \]
	As in the sketch of this proof (given just after the statement of Lemmas \ref{lem::uipart1} and \ref{lem::uipart2}) we have put deterministic bounds on $\tilde{f}_{\delta,2\lex(x)}^\beta(x)$, $\tilde{f}^\beta_{\delta,2\lex(y)}(y)$ and $\e^{\tilde{g}_{\delta,2\lex(y)}(y)}$, and integrated over $\e^{\tilde{g}_{\delta,2\lex(x)}(x)}$.

	Hence we can bound the integral (\ref{eqn::je2}), restricted to the set $x\in \Os$, $y\in \Os\setminus B(x,3\eps)$, by a multiple of  
	\[ \int_{x\in \Os}\int_{y\notin B(x,3\eps)} \frac{1}{|x-y|^2} F_{R}(-\log|x-y|) \, dy \, dx \leq C \int_{x\in \Os}\int_{0}^{\log(1/\eps)} F_{R}(u) du \, dx.  \]
	Since $F_{R}$ is integrable we see that this is uniformly bounded in $\eps$. 
	
	Finally, we must deal with the integral over the set $x\in \Os$, $y\in B(x,3\eps)$. By the same reasoning as above (although now we do not need to do any conditioning, since $\delta(x,y)=\eps$) we see that the integrand on this region is less than some constant times $\eps^{-2}F_{R}(\log(1/\eps))$. That the integral is uniformly bounded in $\eps$ then follows from that fact that $F_{R}(\log(1/\eps))$ is bounded, and that the area of $B(x,3\eps)$ is $\text{O}(\eps^2)$.	\end{proofof}

\subsection{Convergence}\label{sec::convergence}

We now need to show that $D^\beta_\eps(\Os)$ converges (in some sense) as $\eps\to 0$. To do this, we define the change of measure 
\begin{equation}\label{eqn::com} \frac{d\mathbb{Q}^{\beta,\eps}}{d\mathbb{P}}= \frac{D_\eps^\beta(\Os)}{\mathbb{E}[D_\eps^\beta(\Os)]}\end{equation}
for each $\eps>0$. Note that this is not a martingale change of measure, but it is well defined for each $\eps>0$. We will prove the following lemma (from now on we drop the dependence on $\Os$ from our notation for compactness.)

\begin{proposition}\label{lemma::conv_quotient}
	For each fixed $\beta$ and $\eps_0$, and for any $\delta>0$
	\[ \mathbb{Q}^{\beta,\eps} \left[\left|\frac{M_\eps^\beta}{D_\eps^\beta}\sqrt{\log(1/\eps)} - \sqrt{\frac{2}{\pi}} \right|>\delta\right] \to 0  \]
	as $\eps\to 0$. 
\end{proposition}

\begin{remark}
	Since $D_\eps^\beta$ and $M_\eps^\beta$ are close to $D_\eps$ and $M_\eps$ for large $\beta$ (Lemma \ref{lemma::minimum_particle}) and $\mathbb{Q}^{\beta,\eps}$ is defined by a uniformly integrable change of measure (Proposition \ref{lem::ui}) this is almost exactly what we need (recall that by Theorem \ref{theorem::critical_sh_gff} we have $M_\eps \sqrt{\log(1/\eps)}\to \sqrt{\pi/2}\mu'$ as $\eps\to 0$.) Indeed, we will see that the proof of Theorem \ref{proposition::convergence} follows in a straightforward manner once we have completed the proof of Proposition \ref{lemma::conv_quotient}.
\end{remark}

\begin{remark}
	The proof of Proposition \ref{lemma::conv_quotient} follows the general outline of the main proof in \cite{AiSh}. However, the details of each step are somewhat different, and rely on the precise way we have constructed $D_\eps^\beta$. One of the main difficulties is to make exact statements about the behaviour of $h_\eps$ using what we know about the behaviour of $\tilde{h}_\eps$. 
\end{remark}

Before starting the proof, we make a few remarks about the change of measure (\ref{eqn::com}). Define
\[ \hat{\mathbb{Q}}^{\beta,\eps}(dx,dh) = \frac{f_{\eps,2}^\beta(x)\e^{g_{\eps,2}(x)}\I_{\{x\in\Os\}}\I_{L_\eps(x)}\I_{\{f_{\eps,2}^\beta(x)>1\}} dx \, \mathbb{P}[dh]}{\mathbb{E}[D_\eps^\beta]} \]
to be the \emph{rooted measure} on $(h,x)$ where $h$ is a field and $x$ is a point in $\Os$. Introducing this type of measure is a classical tool for dealing with branching processes, that also comes in very useful in the context of Gaussian multiplicative chaos. We have the following description of how the point $x$ and the field $h$ interact under $\hat{\mathbb{Q}}^{\beta,\eps}$:
\begin{itemize}
	\item the marginal law of $h$ under $\hat{\mathbb{Q}}^{\beta,\eps}$ is $\mathbb{E}[D_\eps^\beta]^{-1} D_\eps^\beta d\mathbb{P}$ (i.e. the same law as under $\mathbb{Q}^{\beta,\eps}$);
	\item the marginal law of $x$ under $\hat{\mathbb{Q}}^{\beta,\eps}$, that we shall call $dm^{\beta,\eps}(x)$, is proportional to \[Z_\eps^\beta(x):= \I_{\{x\in\Os\}}\mathbb{E}[f_{\eps,2}^\beta(x)\e^{g_{\eps,2}(x)}\I_{L_\eps(x)}\I_{\{f_{\eps,2}^\beta(x)>1\}}].\]
	\item the conditional law of the field $h$ given the point $x$ is given by \[ \mathbf{Q}_{x}^{\beta,\eps}:=\hat{\mathbb{Q}}^{\beta,\eps}(\cdot \mid x)= (Z_\eps^\beta(x))^{-1} f_{\eps,2}^\beta(x)\e^{g_{\eps,2}(x)}\I_{\{x\in\Os\}}\I_{L_\eps(x)}\I_{\{f_{\eps,2}^\beta(x)>1\}}  \, d\mathbb{P}. \]
	\item the conditional law of the point $x$ given the field $h$ is proportional to \[f_{\eps,2}^\beta(x)\e^{g_{\eps,2}(x)}\I_{\{x\in\Os\}}\I_{L_\eps(x)}\I_{\{f_{\eps,2}^\beta(x)>1\}} dx.\]
\end{itemize}
Also note that
\[\frac{d\mathbf{Q}_x^{\beta,\eps}}{d\mathbf{\tilde{Q}}_x^{\beta,\eps}}=\frac{\tilde{Z}^\beta_\eps(x)}{Z^\beta_\eps(x)}\frac{f_{\eps,2}^\beta(x)}{\tilde{f}_{\eps,2\lambda_\eps}^\beta(x)}\I_{{\{f_{\eps,2}^\beta(x)\geq 1}\}}
\]
where we recall from the proof of Lemma \ref{lem::uipart1} that under $\mathbf{\tilde{Q}}_x^{\beta,\eps}$ the process \[\{(\tilde{f}_{e^{-u},2\lambda_\eps}^\beta(x)-\rho_{e^{-u}}^\eps(x)\,;\, u\in[\log(1/\eps_0),\log(1/\eps)]\}\] has the law of a 3d Bessel process, whose starting point is also biased (and a.s. positive.)
In fact, one of the key ideas in the proof of Proposition \ref{lemma::conv_quotient} will be to say that $f_\eps^\beta(x)$ under $\mathbf{Q}_x^{\beta,\eps}$ also behaves essentially like a Bessel process. As a warm up, let us first prove the following:

\begin{lemma}
	\begin{equation}\label{eqn::Z_ratio} \frac{Z_\eps^\beta(x)}{\tilde{Z}_\eps^\beta(x)}\to 1\end{equation}
	uniformly in $x$ as $\eps\to 0$.  
\end{lemma}This justifies in some sense that the measures $\mathbf{Q}_x^{\beta,\eps}$ and $\mathbf{\tilde{Q}}_x^{\beta,\eps}$ are similar for small $\eps$, and is a result we will use many times.

\begin{proof} We consider the ratio \begin{align*}Z_\eps^\beta(x)/\tilde Z_\eps^\beta(x) & =(\tilde{Z}_\eps^\beta)^{-1} \mathbb{E}[f_{\eps,2}^\beta(x)\e^{g_{\eps,2}(x)}\I_{L_\eps(x)}\I_{\{f_{\eps,2}^\beta(x)> 1\}}] \\ &= \mathbf{\tilde Q}_{x}^{\beta,\eps}\left[ (f_{\eps,2}^\beta(x)/\tilde{f}_{\eps,2\lambda_\eps}^\beta(x)) \I_{\{f_{\eps,2}^\beta(x)>1\}}\right].\end{align*} To show this converges to 1 we write, using decomposition (\ref{eqn::fdecomp}),
	\[ \frac{f_{\eps,2}^\beta(x)}{\tilde{f}_{\eps,2\lex}^\beta(x)} = 1+\text{o}(1)+\frac{\text{O}(1)- Y_\eps(x)}{\tilde{f}_{\eps,2\lex}^\beta(x)}.\] Then by exactly the same Cauchy--Schwarz argument as in the proof of Lemma \ref{lem::uipart1}, it is enough to show that  \begin{equation} \label{eqn::probto0}\mathbf{\tilde{Q}}_x^{\beta,\eps}[f_{\eps,2}^\beta(x)\leq 1]\to 0\end{equation} uniformly in $x$. Since $\tilde{f}_{\eps,2\lex}^\beta(x)$ is close to $\sqrt{\log(1/\eps)}$ with high probability under $\tilde{\mathbf{Q}}_x^{\beta,\eps}$, and $f^\beta_{\eps,2}(x)=\lex(x)\tilde{f}^\beta_{\eps,2\lex}(x)+\text{O(1)} + Y_\eps(x)$, it is sufficient to control the tails of $Y_\eps(x)$ under $\tilde{\mathbf{Q}}_x^{\beta,\eps}$. For this we observe that, by Cauchy--Schwarz again,
	\begin{equation}\label{eqn::yprobsqtilde} \mathbf{\tilde Q}_x^{\beta,\eps}[\I_{\{Y_\eps(x)>a\}}]^2\leq (\tilde{Z}_\eps^\beta(x))^{-1}\mathbf{\tilde{Q}}_x^{\beta,\eps}[\tilde{f}^\beta_{\eps,2\lambda_\eps}(x)]\, \mathbb{E}[\I_{\{Y_\eps(x)>a\}}\e^{\tilde{g}_{\eps,2\lambda_\eps}(x)}\e^{g_{\eps,2}^Y(x)}]\leq \log(1/\eps)\e^{-ka}
	\end{equation}
	for some $k$, since $Y_\eps(x)$ is Gaussian with bounded variance under $\mathbb{P}$, $\tilde{f}_{\eps,2\lambda_\eps}(x)^\beta$ has the law of a 3d-Bessel process at time $\log(1/\eps)$ under $\tilde{\mathbf{Q}}_x^{\beta,\eps}$, and $\tilde{Z}_\eps^\beta(x)$ is bounded below by a constant times the probability that a Brownian motion stays above $-\beta+1$ up to time $\log(1/\eps)$. This allows us to conclude.
\end{proof}

\begin{proofof}{Proposition \ref{lemma::conv_quotient}} Our strategy to prove Proposition \ref{lemma::conv_quotient} is to show the following two things:
	\begin{equation}
	\label{eqn::firstmoment} \mathbb{Q}^{\beta,\eps}\left[\frac{M_\eps^\beta}{D_\eps^\beta}\right]=\sqrt{\frac{2}{\pi \log(1/\eps)}}+\text{o}\left(\frac{1}{\sqrt{\log(1/\eps)}}\right)\;\text{as}\; \eps\to 0;\; \text{and}\end{equation} 
	\begin{equation}\label{eqn::secondmoment}\mathbb{Q}^{\beta,\eps}\left[\left(\frac{M_\eps^\beta}{D_\eps^\beta}\right)^2\right]\leq\frac{2}{\pi \log(1/\eps)}+\text{o}\left(\frac{1}{\log(1/\eps)}\right) \;\text{as}\; \eps\to 0.\end{equation}

	The result then follows using Jensen's and Markov's inequalities. (\ref{eqn::firstmoment}) is relatively straightforward. Observe that, by the discussion preceeding this proof, we have 
	\begin{equation}
	\label{eqn::cond_viewpoint} \frac{M_\eps^\beta}{D_\eps^\beta}=\hat{\mathbb{Q}}^{\beta,\eps}\left[ \frac{1}{f_{\eps,2}^\beta(x)}\mid h \right].
	\end{equation}
	This means that \[\mathbb{Q}^{\beta,\eps}\left[M_\eps^\beta/D_\eps^\beta\right]=\hat{\mathbb{Q}}^{\beta,\eps}\left[M_\eps^\beta/D_\eps^\beta\right]=\hat{\mathbb{Q}}^{\beta,\eps}[f_{\eps,2}^\beta(x)^{-1}]= \int_{\Os} \mathbf{Q}_x^{\beta,\eps}[f_{\eps,2}^\beta(x)^{-1}] \, dm^{\beta,\eps}(x),\] which is a useful representation, because we can write
	\[\mathbf{Q}_x^{\beta,\eps}[f_\eps^\beta(x)^{-1}]=\frac{\tilde{Z}^\beta_\eps(x)}{Z^\beta_\eps(x)}\mathbf{\tilde{Q}}_x^{\beta,\eps}[\tilde{f}_{\eps,2}^\beta(x)^{-1}\I_{\{f_{\eps,2}^\beta(x)>1\}}]\]for each $x\in \Os$. The first moment estimate then follows by (\ref{eqn::probto0}), \eqref{eqn::Z_ratio} and Lemma \ref{lemma::bessel}, parts (1) and (2).

	(\ref{eqn::secondmoment}) is rather more difficult, and requires several steps.

	\textbf{Step 1}: \emph{We show that restricting to an event of high probability under $\hat{\mathbb{Q}}^{\beta,\eps}$ does not affect our second moment too much. That is, we show that if we can find a sequence of events $E_\eps=E_\eps(x)$ with $\mathbb{\hat Q}^{\beta,\eps}[E_\eps]\to 1$ and \begin{equation}\label{eqn::condstep2}\qbeh{\frac{M_\eps^\beta}{D_\eps^\beta}\frac{1}{f_{\eps,2}^\beta(x)}\I_{E_\eps}}\leq \frac{2}{\pi \log(1/\eps)}+\text{\emph{o}}\left(\frac{1}{\log(1/\eps)}\right),\end{equation}
		then this will prove (\ref{eqn::secondmoment}).}
	
	To see how this implies (\ref{eqn::secondmoment}), take such an event $E_\eps$ and write 
	\begin{align*}\mathbb{Q}^{\beta,\eps}\left[\left(\frac{M_\eps^\beta}{D_\eps^\beta}\right)^2\right]& =\hat{\mathbb{Q}}^{\beta,\eps}\left[\left(\frac{M_\eps^\beta}{D_\eps^\beta}\right)^2\right] \\ &=\hat{\mathbb{Q}}^{\beta,\eps}\left[\frac{M_\eps^\beta}{D_\eps^\beta} \hat{\mathbb{Q}}^{\beta,\eps}[ f_{\eps,2}^\beta(x)^{-1}\I_{E_\eps} \mid h]\right]+\hat{\mathbb{Q}}^{\beta,\eps}\left[\frac{M_\eps^\beta}{D_\eps^\beta} \hat{\mathbb{Q}}^{\beta,\eps}[ f_{\eps,2}^\beta(x)^{-1}\I_{E_\eps^c} \mid h]\right] \\
	&= \hat{\mathbb{Q}}^{\beta,\eps}\left[\frac{M_\eps^\beta}{D_\eps^\beta}\frac{1}{f_{\eps,2}^\beta(x)}\I_{E_\eps}\right]+\hat{\mathbb{Q}}^{\beta,\eps}\left[\frac{M_\eps^\beta}{D_\eps^\beta} \hat{\mathbb{Q}}^{\beta,\eps}[ f_{\eps,2}^\beta(x)^{-1}\I_{E_\eps^c} \mid h]\right]. 
	\end{align*}
	We would like to show that the second term in the final expression is $\text{o}(\log(1/\eps)^{-1})$. For this, it is enough by Cauchy--Schwarz, to show that 
	\begin{itemize}
		\item $\qbeh{(M_\eps^\beta/D_\eps^\beta)^2}=\text{O}(\log(1/\eps)^{-1})$, and 
		\item $\qbeh{\xi_\eps^2}=\text{o}(\log(1/\eps)^{-1})$ where $\xi_\eps:= \qbeh{f_{\eps,2}^\beta(x)^{-1}\I_{E_\eps^c}\mid h}$.
	\end{itemize} 
	
	We deal with each point in turn. For the first point, note that by conditional Jensen's inequality we have
	\[ \qbeh{(M_\eps^\beta/D_\eps^\beta)^2}\leq \qbeh{f_{\eps,2}^\beta(x)^{-2}} = \int_{\Os} \qbex{f_{\eps,2}^\beta(x)^{-2}} \, dm^{\beta,\eps}(x)\]
	and then by changing measure and rearranging as usual, we can write 
	\[ \qbex{f_{\eps,2}^\beta(x)^{-2}}=(\tilde{Z}_\eps^\beta(x)/Z^\beta_\eps(x))\, \ebex{\frac{1}{\tilde{f}_{\eps,2\lambda_\eps}^\beta(x)f_{\eps,2}^\beta(x)}  \I_{\{f_{\eps,2}^\beta(x)>1\}}} .\]
	To show that this is O$(\log(1/\eps)^{-1})$ we need to be a little bit careful, although heuristically it is clear from the fact that $\tilde{f}$ is a Bessel process and $Y_\eps(x)$ is small.
	The way to make this precise is to consider the expectation on the ``good'' event, \[\{\tilde{f}^\beta_{\eps,2\lex}(x)>\log(1/\eps)^{1/4}\} \cap \{Y_\eps(x)<(1/2)\log(1/\eps)^{1/4}\}\] and its complement separately. On the good event we have that $f_{\eps,2}^\beta(x)\geq c\tilde{f}_{\eps,\lex}^\beta(x)$ for some constant $c$, and so the expectation is $\text{O}(\log(1/\eps)^{-1})$ by Lemma \ref{lemma::bessel}, part (2). On the bad event, we use (\ref{eqn::yprobsqtilde}) and Lemma \ref{lemma::bessel}, part (5), to see that the expectation is also $\text{O}(\log(1/\eps)^{-1})$.

	Now we treat the second point. By Jensen's inequality, and for any $a>0$, we have
	\begin{align*}\qbeh{\xi_\eps^2} & \leq \qbeh{f_\eps^\beta(x)^{-2}\I_{E_\eps^c}} \\
		& = \qbeh{\frac{\I_{E_\eps^c}}{f_{\eps,2}^\beta(x)^{2}} \I_{\{f_{\eps,2}^\beta(x)\geq a \sqrt{\log(1/\eps)}\}}}+
	\qbeh{\frac{1}{f_{\eps,2}^\beta(x)^{2}} \I_{\{f_{\eps,2}^\beta(x)< a \sqrt{\log(1/\eps)}\}}}.
	\end{align*}
	
	It is clear by definition that the first term is less than $\mathbb{\hat{Q}}^{\beta,\eps}[E_\eps^c]/(a^2 \log(1/\eps))$, and for  
	the second term, we write it as
	\begin{equation}\label{ec2ndmoment} \int_{\Os} \frac{\tilde{Z}_\eps^\beta(x)}{Z_\eps^\beta(x)} \ebex{\frac{1}{\tilde{f}_{\eps,2\lex}^\beta(x) f_{\eps,2}^\beta(x)} \I_{\{f_{\eps,2}^\beta(x)>1\}}\I_{\{f_{\eps,2}^\beta(x)< a \sqrt{\log(1/\eps)}\}}} 
	\, dm^{\beta,\eps}(x). \end{equation}
	Similarly to before, we consider the expectation on the event  \[\{Y_\eps(x)<(a/2)\log(1/\eps)^{1/4}\}\cap\{\tilde{f}_{\eps,2\lex}^\beta(x)>a \log(1/\eps)^{1/4} \}\] and its complement separately. This allows us to bound (\ref{ec2ndmoment}) by some constant times $a\log(1/\eps)^{-1}+\sqrt{\log(1/\eps)}\exp(-(ak/2)\log(1/\eps)^{1/4})$, where $k>0$ is the constant from (\ref{eqn::yprobsqtilde}). Thus, \[\limsup_{\eps\to 0} \log(1/\eps)\, \qbeh{\xi_\eps ^2}\leq Ca \] for any $a>0$ and some fixed finite $C$. Taking $a\to 0$, this allows us to conclude step 1.

	\textbf{Step 2}: \emph{We define the event $E_\eps$, and set up the scales for the multiscale argument we will use.}

	To do this we let $r_\eps>\eps$ be a sequence with 
	\begin{equation}
	\label{eqn::conditions_k}
	\frac{\log(1/r_\eps)}{\log(1/\eps)^{1/3}}\to \infty \;\; \text{and} \;\; \frac{\log(1/r_\eps)}{\log(1/\eps)^{1/2}}\to 0 \end{equation} 
	as $\eps\to 0$ (so $r_\eps$ is tending to $0$ much slower than $\eps$). Given this, we break up $D$ and $M$ as \[D_\eps^\beta=D_\eps^{\beta,in}+D_\eps^{\beta,out} \quad \text{and} \quad M_\eps^\beta=M_\eps^{\beta,in}+M_\eps^{\beta,out}\] where the subscript $in$ refers to the integral inside $B(x,r_\eps)$ and the subscript $out$ refers to the integral outside of it.
	
	The basic idea is that $D_\eps^{\beta,in}$ and $M_\eps^{\beta,in}$ will be small with high probability (this will be part of the definition of $E_\eps$) and on this event, $M_\eps^\beta/D_\eps^\beta$ will be close to $ M_\eps^{\beta,out}/D_\eps^{\beta,out}$. Heuristically, this occurs with high probability because the limits of $M_\eps$ and $D_\eps$ should be atomless measures, and $r_\eps$ is tending to $0$. Next, we claim that  $M_\eps^{\beta,out}/D_\eps^{\beta,out}$ is essentially independent of $f_{\eps,2}^\beta(x)$. This is because $r_\eps$ is much larger than $\eps$ and $f$ is approximately a (time changed) Bessel process, so its value at time $\eps$ is basically independent of its value at time $(r_\eps-\eps)$. From here (\ref{eqn::condstep2}) follows, since we already know that  $M_\eps^\beta/D_\eps^\beta$ and $f_{\eps,2}^\beta(x)^{-1}$ have (the same) expectation, of the right order.
	
	We now choose our event $E_\eps$, according to this plan. To do this, we first have to observe that, by the Markov property of the field, $Y_\eps(x)=h_\eps(x)-\lex(x)\tilde{h}_\eps(x)$ can be written as \[Y_\eps(x):={Y}^1_\eps(x)+Y^2_\eps(x),\] where $Y^2_\eps$ is independent of $h|_{D\setminus B(x,r_\eps-\eps)}$ and $Y^1_\eps$ is measurable with respect to $h|_{D\setminus B(x,r_\eps-\eps)}$ (see the proof of Lemma \ref{lem::uiaugpart2} for a more detailed explanation.) Given this definition, we set $E_\eps=E_\eps^1\cap E_\eps^2$ where
	\begin{align*} E_\eps^1&=\{D_\eps^{\beta,in}\leq \log(1/\eps)^{-2}\}\\
	E_\eps^2 &= \{\tilde{f}_{u,2\lex}^\beta(x)\in [\sqrt[3]{\log(1/u)},\log(1/u)] \; \forall u\in[r_\eps-\eps, r_\eps]\}\cap \{Y^1_\eps(x)<(\log(1/\eps))^{1/4} \}  
	\end{align*}
	(we will prove that $\qbeh{E_\eps}\to 1$ later on.) We remark that the event $E_\eps^2$ here is needed for the ``independence'' step.

	\textbf{Step 3}: \emph{We split the left hand side of (\ref{eqn::condstep2}) into two parts: one concerning the measures restricted to $\Os \cap B(x,r_\eps)$, and one concerning the measures restricted to $\Os \cap (B(x,r_\eps)\setminus B(x,\eps))$. We show that the first of these is negligible compared to $\log(1/\eps)$.}

	More precisely, we write
	\begin{equation}\label{eqn::secondmoment_factorised} \qbeh{\frac{M_\eps^\beta}{D_\eps^\beta}\frac{1}{f_{\eps,2}^\beta(x)}\I_{E_\eps}}\leq \qbeh{\frac{M_\eps^{\beta,in}}{D_\eps^\beta}\frac{1}{f_{\eps,2}^\beta(x)}\I_{E_\eps}}+\qbeh{\frac{M_\eps^{\beta,out}}{D_\eps^{\beta,out}}\frac{1}{f_{\eps,2}^\beta(x)}\I_{E_\eps^2}}.
	\end{equation}
	Then by definition we have that $M_\eps^{\beta,in}\leq D_{\eps}^{\beta,in}$, and so on the event $E_\eps$ it holds that $M_\eps^{\beta,in}\leq \log(1/\eps)^{-2}$. Moreover, we know that $f_{\eps,2}^\beta(x)$ is greater than $1$ under $\hat{\mathbb{Q}}^{\beta,\eps}$ and also that $\hat{\mathbb{Q}}^{\beta,\eps}[1/D_\eps^\beta]=\mathbb{E}[D_{\eps}^\beta]^{-1}=\text{O}(1)$. Thus, the first term is $\text{o}(\log(1/\eps)^{-1})$ and we need only treat the second term.

	\textbf{Step 4}: \emph{We condition on the field outside of $B(x, {r_\eps-\eps})$ in order to factorise the second term on the right-hand side of (\ref{eqn::secondmoment_factorised}). We show that the conditional expectation of $f_{\eps,2}^\beta(x)^{-1}$ is of order $\sqrt{2/\pi \log(1/\eps)}(1+\text{\emph{o}}(1))$ uniformly on $E_\eps^2$.}

	More precisely, we condition on $\F_{r_\eps-\eps}$, the $\sigma$-algebra generated by the point $x$ and the field $h$ restricted to $D\setminus B(x,r_\eps-\eps)$. Then $M_\eps^{\beta,out},$ $D_\eps^{\beta,out}$ and $E_\eps^2$ are measurable with respect to $\F_{r_\eps-\eps}$, meaning that the second term on the right-hand side of (\ref{eqn::secondmoment_factorised}) is equal to \[ \qbeh{\frac{M_\eps^{\beta,out}}{D_\eps^{\beta,out}}\I_{E_\eps^2} \qbeh{\frac{1}{f_{\eps,2}^\beta(x)}\mid \F_{r_\eps-\eps}}}.\]
	
	Now we write
	\begin{align*} \qbeh{\frac{1}{f_{\eps,2}^\beta(x)} \, \Big| \, \F_{r_\eps-\eps}} & =\qbex{\frac{1}{f_{\eps,2}^\beta(x)} \, \Big| \, h|_{D\setminus B(x,r_\eps-\eps)}}  \\ 
		& = \frac{\ebex{\frac{1}{\tilde{f}_{\eps,2\lex}^\beta(x)} \I_{\{f_{\eps,2}^\beta(x)\geq 1\}} \, \Big| \,  h|_{D\setminus B(x,r_\eps-\eps)}}}{\ebex{ \frac{f_{\eps,2}^\beta(x)}{\tilde{f}_{\eps,2\lex}^\beta(x)} \I_{\{f_{\eps,2}^\beta(x)\geq 1\}}  \, \Big| \, h|_{D\setminus B(x,r_\eps-\eps)}}}. \end{align*}
	We will show that, \emph{on the event $E_\eps^2$}, the numerator in the final expression is less than or equal to $\sqrt{2/(\pi \log(1/\eps))}+\text{o}(\log(1/\eps)^{-1/2})$ and the denominator is $1+\text{o}(1)$. To do this we observe that, on $E_\eps^2$ and under the conditional law $\mathbf{\tilde{Q}}_{x}^{\beta,\eps}[\, \cdot \mid h_{D\setminus B(x,r_\eps-\eps)}]$:
	\begin{itemize}
		\item  $\tilde{f}_{\eps,2\lambda_\eps}^\beta(x)$ has the law of a Bessel process started from a position in \\ $[\sqrt[3]{\log(1/(r_\eps-\eps))},\log(1/(r_\eps-\eps))]$ and evaluated at time $\log(r_\eps-\eps)-\log(\eps)$;
		\item by choice of $r_\eps$ and Lemma \ref{lemma::bessel} part (1), this implies that the conditional expectation of $(\tilde{f}_{\eps,2\lex}^\beta(x))^{-1}$ is equal to $\sqrt{2/(\pi \log(1/\eps))}(1+\text{o}(1))$;
		\item the conditional expectation of $|Y^2_\eps|/\tilde{f}_{\eps,2\lex}(x)$ is $\text{o}(1)$, by Cauchy--Schwarz; and 
		\item $|Y^1_\eps(x)|$ times the conditional expectation of $(\tilde{f}_{\eps,2\lex}^\beta(x))^{-1}$ is also o(1), by the second point, and definition of $E_\eps^2$.
	\end{itemize}
	Together these imply that, uniformly on $E_\eps^2$,
	\[\qbeh{1/f_{\eps,2}^\beta(x) \mid \F_{r_\eps-\eps}}\leq \sqrt{2/(\pi \log(1/\eps))}(1+\text{o}(1)).\] 
	
	\textbf{Step 5}: \emph{We show that  $\qbeh{\frac{M_\eps^{\beta,out}}{D_\eps^{\beta,out}}\I_{E_\eps^2}}$ is also bounded above by $\sqrt{2/(\pi \log(1/\eps))}(1+\text{\emph{o}}(1))$. This completes the proof of (\ref{eqn::condstep2}) for our choice of $E_\eps$}.

	This is the most delicate step. To do this, we define yet another event  
	\[E_\eps^3 = \{\tilde{f}_{u,2\lex}^\beta(x)\geq \sqrt[6]{\log(1/r_\eps)} \; \forall u\in [\eps,r_\eps]\}.\]
	\textbf{Step 5(i):} We will first show that \begin{equation}\label{eqn::e2e3sim}\qbeh{\frac{M_\eps^{\beta,out}}{D_\eps^{\beta,out}}\I_{E_\eps^2}\I_{E_\eps^3}}\geq \qbeh{\frac{M_\eps^{\beta,out}}{D_\eps^{\beta,out}}\I_{E_\eps^2}}\left(1+\text{o}(1)\right)\end{equation}
	so we can instead consider the term on the left-hand side (which turns out to be easier to deal with.) To see why (\ref{eqn::e2e3sim}) is true we again condition on $\F_{r_\eps-\eps}$. This gives us that 
	\[ \qbeh{\frac{M_\eps^{\beta,out}}{D_\eps^{\beta,out}}\I_{E_\eps^2}\I_{E_\eps^3}}\geq \qbeh{\frac{M_\eps^{\beta,out}}{D_\eps^{\beta,out}}\I_{E_\eps^2} \qbeh{E_\eps^3\mid \F_{r_\eps-\eps}}}\]
	where by changing measure as in step 4 we have
	\[ \qbeh{(E_\eps^3)^c \mid \F_{r_\eps-\eps}} = \frac{\ebex{ \frac{f_{\eps,2}^\beta(x)}{\tilde{f}_{\eps,2\lex}^\beta(x)} \I_{\{f_{\eps,2}^\beta(x)\geq 1\}} \I_{(E_\eps^3)^c} \, \Big| \, h|_{D\setminus B(x,r_\eps-\eps)}}}{\ebex{ \frac{f_{\eps,2}^\beta(x)}{\tilde{f}_{\eps,2\lex}^\beta(x)} \I_{\{f_{\eps,2}^\beta(x)\geq 1\}}  \, \Big| \, h|_{D\setminus B(x,r_\eps-\eps)}}}. \]
	We already know that the denominator is $1+\text{o}(1)$ uniformly on $E_\eps^2$ by our previous discussion. In fact, the numerator is also $\text{o}(1)$ uniformly on $E_\eps^2$. To show this, again using our observations from step 4, it is enough for us to prove that \[\ebex{E_\eps^3\mid h|_{D\setminus(B(x,r_\eps-\eps))}}\to 1\] uniformly on $E_\eps^2$. For this we again use the fact that, under this conditional law, the process $\tilde{f}_{u,2\lex}^\beta(x)$ for $u\leq r_\eps-\eps$ is a time-changed Bessel process (plus a small deterministic fluctuation $\rho_{u}^\eps(x)$) starting from a position in $[\sqrt[3]{\log(1/(r_\eps-\eps))},\log(1/(r_\eps-\eps))]$. Thus we need to calculate the probability that such a Bessel process, and we can clearly forget about the fluctuations, remains greater than $ \log(1/r_\eps)$ up to time $\log(1/\eps)-\log(1/(r_\eps-\eps))$. It is clear that this probability is smallest if we take the starting point $x_0$ to be $\sqrt[3]{\log(1/(r_\eps-\eps))}$. In this case we have, writing $\mathbf{Q}_{y}$ for the law of a Bessel process started at $y$, and by scaling, that 
	\begin{align*}
	& \mathbf{Q}_{x_0}(X_t \geq \sqrt[6]{\log(1/r_\eps)}\;\forall t\in[0,\log(1/\eps)-\log(1/(r_\eps-\eps))]) \\ \geq \;\; & \mathbf{Q}_1(X_t \geq \frac{\sqrt[6]{\log(1/r_\eps)}}{\sqrt[3]{\log(1/(r_\eps-\eps))}} \; \forall t\in[0,\infty]).
	\end{align*}
	Taking $\eps\to 0$ we see that this converges to $1$: the probability that a 3d Bessel process started at $1$ never hits $0$. 
	\medbreak
	
	\noindent \textbf{Step 5(ii):} Having done this, we can now consider the left-hand side of (\ref{eqn::e2e3sim}) and instead try to show that this is bounded above by $\sqrt{2/(\pi \log(1/\eps))}(1+\text{o}(1))$. Again this will require a few arguments. We  write  
	\begin{align*} \qbeh{\frac{M_\eps^{\beta,out}}{D_\eps^{\beta,out}}\I_{E_\eps^2 \cap E_\eps^3}} & \leq \qbeh{\frac{M_\eps^{\beta,out}}{D_\eps^{\beta,out}}\I_{E_\eps^2\cap E_\eps^3}\I_{E_\eps^1}\I_{\{D_\eps^\beta>\log(1/\eps)^{-1}\}}} +
	 \\   &\qbeh{\frac{M_\eps^{\beta,out}}{D_\eps^{\beta,out}}\I_{E_\eps^2}\I_{E_\eps^1}\I_{\{D_\eps^\beta \leq \log(1/\eps)^{-1}\}}}+\qbeh{\frac{M_\eps^{\beta,out}}{D_\eps^{\beta,out}}\I_{E_\eps^2\cap E_\eps^3}\I_{(E_\eps^1)^c}}
	\end{align*}
	and will show that the first term is of the order we want, and the second and third are negligible. Indeed, it is clear that the first term is less than or equal to 
	\[ \qbeh{\frac{M_\eps^\beta}{D_\eps^\beta}}(1+\text{o}(1)) \leq \sqrt{\frac{2}{\pi \log(1/\eps)}}(1+\text{o}(1))\]
	by definition of the events (these imply that $D_\eps^{\beta}/D_\eps^{\beta,out}=1+\text{o}(1)$) and our previous estimate (\ref{eqn::firstmoment}) for the first moment. Moreover by Markov's inequality for $(1/D_\eps^\beta)$, and the fact that $M_\eps^{\beta,out}/D_\eps^{\beta,out}\leq 1$, the second is $\text{o}(\sqrt{\log(1/\eps)}^{-1})$. 
	\medbreak
	
	\noindent\textbf{Step 5(iii):} So we are left to deal with the third term, which we would also like to show is $\text{o}(\sqrt{\log(1/\eps)}^{-1})$. To do this, we first observe that we can bound it above by
	\begin{equation}\label{eqn::small_ball_estimates}\qbeh{E_\eps^2\cap E_\eps^3\cap (E_\eps^1)^c}\leq\qbeh{E_\eps^3\cap \left\{D_\eps^{\beta,in}>\frac{1}{\log(1/\eps)^2}\right\}}.\end{equation}
	Our strategy here will be to use Markov's inequality for $D_\eps^{\beta,in}$. For this we have to calculate the $\hat{\mathbb{Q}}^{\beta,\eps}$ expectation of $D_\eps^{\beta,in}$, which by definition is the same as calculating the $\mathbb{P}$ expectation of $D_\eps^\beta \times D_\eps^{\beta, in}$. Thus, we can use similar techniques to those in the proof of uniform integrability (Lemma \ref{lem::uipart1}), where we calculated the $\mathbb{P}$ expectation of $(D_\eps^\beta)^2$ on a ``good" event. 
	
	As we did there, we will break up $D_\eps^{\beta, in}$ into two parts: the integral over $B(x,3\eps)$, and the rest. To deal with the integral over $B(x,3\eps)$ we define a further event $E_\eps^4$ (which has high probability) and on which $f_{\eps,2}^\beta$ is close to $\sqrt{\log(1/\eps)}$. Crude estimates on this event, using that $|B(x,3\eps)|=\text{O}(\eps^2)$, then give the desired expectation. To deal with the integral over $B(x,r_\eps)\setminus B(x,3\eps)$ we need to be more careful. Here we use the definition of $E_\eps^3$, which allows us to control the value of $\tilde{f}_{\eps,2\lex}^\beta(x)$ at all times between $r_\eps$ and $\eps$. Applying a decorrelation argument similar to that in the proof of Lemma \ref{lem::uipart1} then allows us to reach the desired conclusion.
	
	Let us first define $E_\eps^4$, using the following lemma: 
	
	\begin{lemma}\label{lem::Eeps4}
		For all $p\in(1/4,1/2)$ we have \[E_\eps^4=\{f_{\eps,2}^\beta(x)\in[\log(1/\eps)^{1/2-p},\log(1/\eps)^{1/2+p}]\}\] satisfies $\qbeh{(E_\eps^4)^c}=\text{\emph{o}}(\log(1/\eps)^{-1/2})$. 
	\end{lemma}
	
	\begin{proofof}{Lemma \ref{lem::Eeps4}} This is possible because for any $p>0$
		\[ (E_\eps^4)^c \subset \{\tilde f_{\eps,2\lex}^\beta(x)\notin[2\log(1/\eps)^{1/2-p},\frac{1}{2}\log(1/\eps)^{1/2+p}]\} \cup \{|Y_\eps(x)|> \log(1/\eps)^{1/2-p}\}\] and then \begin{align*}\qbeh{(E_\eps^4)^c} & \leq \int_\Os \qbex{|Y_\eps(x)|> \log(1/\eps)^{1/2-p}}\\ &+\qbex{\tilde f_{\eps,2\lex}^\beta(x)\notin[2\log(1/\eps)^{1/2-p},
			\frac{1}{2}\log(1/\eps)^{1/2+p}]} \, dm^{\beta,\eps}(x). \end{align*} It is easy to see (using the definition of the measure $\mathbf{Q}_x^{\beta,\eps}$) that the first probability inside the integral decays exponentially in $\log(1/\eps)$. For the second, we can write it as 
		\[ \frac{Z_\eps^\beta(x)}{\tilde{Z}_\eps^\beta(x)}\ebex{\frac{f_{\eps,2}^\beta(x)}{\tilde{f}_{\eps,2\lex}^\beta(x)}\I_{\{\tilde{f}_{\eps,2\lex}^\beta(x)\notin[2\log(1/\eps)^{1/2-p},\frac{1}{2}\log(1/\eps)^{1/2+p}]\}}}. \]  which by Cauchy--Schwarz is less than or equal to 
		\begin{align*} & \lex(x) \ebex{\I_{\{\tilde{f}_{\eps,2\lex}^\beta(x)\notin[2\log(1/\eps)^{1/2-p},\frac{1}{2}\log(1/\eps)^{1/2+p}]\}}} + \\
		&
		\ebex{\frac{(Y_\eps(x)+\text{O}(1))^2}{\tilde{f}_{\eps,2\lex}^\beta(x)}}^{1/2}\ebex{\frac{\I_{\{\tilde{f}_{\eps,2\lex}^\beta(x)\notin[2\log(1/\eps)^{1/2-p},\frac{1}{2}\log(1/\eps)^{1/2+p}]\}}}{\tilde{f}_{\eps,2\lex}^\beta(x)}}^{1/2}.
		\end{align*}
		A standard Bessel calculation (if $(\beta_t; \,t\ge 0)$ is a 3d Bessel process then \\ $\mathbb{E}[(\beta_t)^{-1}\I_{\beta_t\le t^{(1/2-p)}}]\lesssim t^{-1/2-2p}$) plus the fact that \[\ebex{(\text{O}(1)+Y_\eps(x))^2/\tilde{f}_{\eps,2\lex}^\beta(x)}=\text{O}(1)\] (seen by changing back to the measure $\mathbb{P}$), gives that this is $\text{o}(\log(1/\eps)^{-1/2})$ for $p\in(1/4,1/2)$. 
	\end{proofof}
	\medbreak
	
	Using this new event $E_\eps^4$, and Markov's inequality, we next bound the right-hand side of (\ref{eqn::small_ball_estimates}) above by 
	\begin{align} \label{markovbig}  & \text{o}(\log(1/\eps)^{-1/2}) + \log(1/\eps)^2  \,\qbeh{ \I_{E_\eps^4} \int_{w\in B(x,3\eps)} f_{\eps,2}^\beta(w)\I_{L_\eps(w)}\I_{\{f_{\eps,2}^\beta(w)>1\}}\e^{g_{\eps,2}(w)} \, dw}  \nonumber \\ 
	& + \log(1/\eps)^2 \, \qbeh{ \I_{E_\eps^3} \int_{w\in B(x,r_\eps)\setminus B(x,3\eps)} f_{\eps,2}^\beta(w)\I_{L_\eps(w)}\I_{\{f_{\eps,2}^\beta(w)>1\}}\e^{g_{\eps,2}(w)} \, dw} 
	\end{align}
	where the first term comes from the $\hat{\mathbb{Q}}^{\beta,\eps}$ probability of $E_\eps^4$. Recall that, to conclude, we need to show this whole expression is $\text{o}(\log(1/\eps)^{-1/2})$.
	Let us look at the expectation in the second term. By definition of $\hat{\mathbb{Q}}^{\beta, \eps}$ this is equal to $\mathbb{E}[D_\eps^\beta]^{-1}$ times
	\begin{align*}
	& \int_{x\in\Os} \int_{w\in B(x,3\eps)}\E{\I_{E_\eps^4}f_{\eps,2}^\beta(w)\I_{L_\eps(w)}\I_{\{f_{\eps,2}^\beta(w)>1\}}\e^{g_{\eps,2}(w)}f_{\eps,2}^\beta(x)\I_{L_\eps(x)}\I_{\{f_{\eps,2}^\beta(x)>1\}}\e^{g_{\eps,2}(x)}}\\
	&  \int_{x\in \Os} \int_{w\in B(x,3\eps)} \eps^{-2} \log(1/\eps)^{1/2+p}\e^{-2\log(1/\eps)^{1/2-p}} \E{f_{\eps,2}^\beta(w)\I_{L_\eps(w)}\I_{\{f_{\eps,2}^\beta(w)>1\}}\e^{g_{\eps,2}(w)}}
	\end{align*}
	where the second line follows from the definition of $E_\eps^4$. Hence, the expectation in the second term of (\ref{markovbig}) is less than or equal to $ \text{O}(1) \times \log(1/\eps)^{1/2+p}\e^{-2\log(1/\eps)^{1/2-p}}$, meaning that the term itself is $\text{o}(\log(1/\eps)^{-1/2})$. 
	
	We finish by dealing with the third term of (\ref{markovbig}). Writing $A_{\eps}=B(x,r_\eps)\setminus B(x,3\eps)$, we have that the expectation in this term is equal to 
	\[\mathbb{E}[D_\eps^\beta]^{-1}\iint_{x\in \Os \atop w\in A_\eps}\E{\I_{E_\eps^3}f_{\eps,2}^\beta(w)\I_{L_\eps(w)}\I_{\{f_{\eps,2}^\beta(w)>1\}}\e^{g_{\eps,2}(w)}f_{\eps,2}^\beta(x)\I_{L_\eps(x)}\I_{\{f_{\eps,2}^\beta(x)>1\}}\e^{g_{\eps,2}(x)}}\]
	Then, by exactly the same reasoning used in the proof of Lemma \ref{lem::uipart1}, we can deduce that this is less than or equal to some constant times
	\begin{equation} \label{eqn::oneofthelast}\iint_{x\in \Os \atop w\in A_\eps} \E{ \e^{\tilde{g}_{\delta,2\lex(x)}(x)}\e^{\tilde{g}_{\delta,2\lex(w)}(w)}\I_{\{\tilde{f}_{\delta,2\lex}^\beta(x)\geq \log(1/r_\eps)^{1/6}\}} (\tilde{f}_{\delta,2\lex(x)}^{\beta}(x)+1)(\tilde{f}_{\delta,2\lex(w)}^{\beta}(w)+1)}  \end{equation}
	where $\delta(x,y)=|x-y|/3$. The final observation to make is that, by orthogonal projection, we have 
	\[ \tilde{h}_\delta(w)=\alpha^\delta_{x,w} \tilde{h}_\delta(x)+Z^\delta_{x,w} \]
	for $\alpha^\delta_{x,w} = \cov(\tilde{h}_\delta(x),\tilde{h}_\delta(w))/\var{\tilde{h}_\delta(x)}$ where $Z^{\delta}_{x,w}$ is independent of $\tilde{h}_\delta(x)$; distributed as a centered normal random variable with variance $(1-(\alpha_{x,w}^\delta)^2) \var{\tilde{h}_\delta(w)}$. The proof of this is the same as the proof of Lemma \ref{lemma::compare_average}. Moreover,  by Lemma \ref{lemma::cov_h}, we have $\alpha_{x,w}^\delta=1+\text{O}(\log(1/\delta)^{-1})$ uniformly in $x,w$.
	
	Thus, by conditioning on $\tilde{h}_\delta(x)$, we can calculate that the integrand in (\ref{eqn::oneofthelast}) is less than or equal to a constant times
	\begin{equation*}\label{eqn::expwtox}
	\mathbb{E}[\I_{\{\tilde{f}_{\delta,2\lex}^\beta(x) \geq \log(1/r_\eps)^{1/6}\}}\e^{\tilde{g}_{\delta,2\lex(x)}(x)+\tilde{g}_{\delta,2\lex(w)\alpha_{x,w}^\delta}(x)}  (\tilde{f}_{\delta,2\lex(x)}^{\beta}(x)+1)(\alpha_{x,w}^\delta\tilde{f}_{\delta,2\lex(w)\alpha_{x,w}^\delta}^{\beta}(x)+1)]\end{equation*}
	which by a simple calculation can be bounded again by
	\[ \e^{-(\log(1/r_\eps))^{1/6}}\delta^{-2} \E{\e^{\tilde{g}_{\delta,2\lex(x)}(x)} (1+|\tilde{f}_{\delta,2\lex(x)}^\beta(x)|^2)}
	\]
	times some constant. This last expectation can be estimated by observing that changing measure by $\e^{\tilde{g}_{\delta,2\lex(x)}(x)}$ turns $\tilde{f}_{\delta,2\lex(x)}^\beta(x)$ into a Gaussian random variable with mean $\beta$ and variance $\log(1/\delta)+\text{O}(1)$. Hence the integrand of (\ref{eqn::oneofthelast}) is less than
	\[C\e^{-(\log(1/r_\eps))^{1/6}}\delta^{-2}(1+\log(1/\delta))\]
	for some constant $C$. Integrating over $A_\eps$ gives that the third term of (\ref{markovbig}) is $\lesssim \e^{-\log(1/r_\eps)^{1/6}} \log(1/\eps)^4$, and so we conclude, using our assumption (\ref{eqn::conditions_k}) on $r_\eps$, that (\ref{markovbig}) is of the correct order.

	\textbf{Step 6}: \emph{We show that $\qbeh{E_\eps}\to 1$ as $\eps\to 0$.}

	Firstly, it is clear that $\mathbb{\hat{Q}}^{\beta,\eps}[E_\eps^2]\to 1$. Then since we have already shown, see (\ref{eqn::small_ball_estimates}), that $\qbeh{E_\eps^2\cap E_\eps^3 \cap (E_\eps^1)^c}\to 0$ and that \[\ebex{E_\eps^3\mid h|_{D\setminus(B(x,r_\eps-\eps)}}\to 1\] uniformly on $E_\eps^2$, the claim follows straight away. The proof is complete.
\end{proofof}
\medbreak

It is now relatively simple to show the convergence of $D_\eps(\Os)$. To use Proposition \ref{lemma::conv_quotient}, we must first compare $M_\eps^\beta(\Os)$ and $D_\eps^\beta(\Os)$ with $M_\eps(\Os)$ and $D_\eps(\Os)$.

\begin{lemma}\label{lemma::min_particle}
	We have 
	\begin{align*}\mathbb{P}(C_\beta):=\mathbb{P}\big( & \{\inf_{\eps}\inf_{x\in D} (-\tilde{h}_{\eps}(x) +2\emph{\text{var}}(\tilde{h}_{\eps}(x)))>-(\beta+10)\} \\ & \cap \{\inf_{\eps}\inf_{x\in D} (-h_{\eps}(x) +2\emph{\text{var}}(h_{\eps}(x)))>-(\beta+10)\}\big) \end{align*}
	converges to $0$ as $\beta\to 0$.
\end{lemma}

\begin{proof}
	This is a consequence of Lemma \ref{rmk::min_particle}.
\end{proof}

\begin{remark} \label{rmk::Cb} Note that on the event $C_\beta$ we have $M_\eps^\beta(\Os)=M_\eps(\Os)$ and also $D_\eps^\beta(\Os)=D_\eps(\Os)+\beta M_\eps(\Os)$ for all $\Os\subset D$. 
\end{remark}

We are now ready to prove the main result. \\

\begin{proofof}{Theorem \ref{proposition::convergence}}
	It is enough to show that for $\Os\subset D$ and $\delta>0$ fixed 
	\begin{equation}\label{eqn::convprob} \mathbb{P}\left[ \left|\frac{M_\eps(\Os)}{D_\eps(\Os)}\sqrt{\log(1/\eps)}-\sqrt{\frac{2}{\pi}}\right|>\delta\right]\to 0
	\end{equation}
	as $\eps\to 0$. Then since $\sqrt{\log(1/\eps)}M_\eps(\Os) \to \sqrt{\frac{\pi}{2}}\mu'(\Os)$ in probability, by Theorem \ref{theorem::critical_sh_gff}, we also have $D_\eps(\Os) \to \mu'(\Os)$ in probability. Let us prove (\ref{eqn::convprob}). By Proposition \ref{lemma::conv_quotient} we know that for any $\beta>0$
	\[  \mathbb{\hat{Q}}^{\beta,\eps}\left[ \left|\frac{M^\beta_\eps(\Os)}{D^\beta_\eps(\Os)}\sqrt{\log(1/\eps)}-\sqrt{\frac{2}{\pi}}\right|>\delta\right]\to 0 \]
	as $\eps\to 0$, and we also know that, on the event $C_\beta$, we can compare $M_\eps^\beta, D_\eps^\beta$ with $M_\eps, D_\eps$ by Remark \ref{rmk::Cb}. With this in mind, we bound (\ref{eqn::convprob}) above by  
	\[\mathbb{P}\left[ A_{\beta,\eps}^1 \right]+ \mathbb{P}\left[ A_{\beta,\eps}^2\cap (A_{\beta,\eps}^1)^c\right]
	\]
	where
	\begin{align*}A_{\beta,\eps}^1 &= \left\{\left|\frac{M_\eps(\Os)}{D_\eps(\Os)+\beta M_\eps(\Os)}\sqrt{\log(1/\eps)}-\sqrt{\frac{2}{\pi}}\right|>\frac{\delta}{2}\right\}\; \text{and} \\ A_{\beta,\eps}^2 & =  \left\{\sqrt{\log(1/\eps)}\left|\frac{M_\eps(\Os)}{D_\eps(\Os)+\beta M_\eps(\Os)}-\frac{M_\eps(\Os)}{D_\eps(\Os)}\right|>\frac{\delta}{2}\right\}.\end{align*}
	In fact, the event $A_{\beta,\eps}^2\cap (A_{\beta,\eps}^1)^c$ is deterministically non possible if $\eps$ is small enough. Thus, it is enough to show that $\mathbb{P}(A_{\beta,\eps}^1)$ can be made arbitrarily small by choosing $\beta$ large, and then $\eps$ small. To do this, we observe by Remark \ref{rmk::Cb} that
	\[ \mathbb{P}(A_{\beta,\eps}^1)\leq \mathbb{P}[C_\beta^c]+ \mathbb{P}\left[ \left\{\left|\frac{M^\beta_\eps(\Os)}{D^\beta_\eps(\Os)}\sqrt{\log(1/\eps)}-\sqrt{\frac{2}{\pi}}\right|>\delta/2\right\} \cap C_\beta \right]. \] Furthermore, by the definition of $\mathbb{\hat{Q}}^{\beta,\eps}$, the right-hand side for any $\eta>0$, is less than or equal to
	\[ \mathbb{P}[C_\beta^c]+\mathbb{P}[C_\beta \cap \{D_\eps^\beta <\eta\}]+\frac{\mathbb{E}[D_\eps^\beta]}{\eta} \mathbb{\hat{Q}}^{\beta,\eps}\left[ \left|\frac{M^\beta_\eps(\Os)}{D^\beta_\eps(\Os)}\sqrt{\log(1/\eps)}-\sqrt{\frac{2}{\pi}}\right|>\frac{\delta}{2}\right].
	\]
	Now note that by Markov's inequality 
	\[ \mathbb{P}[C_\beta\cap\{D_\eps^\beta<\eta\}] \leq \mathbb{P}[\sqrt{\log(1/\eps)}M_\eps(\Os)< \eta^{1/4}] + \sqrt{\eta}\, \mathbb{E}[{D_\eps^\beta}]\, \mathbb{\hat{Q}}^{\beta,\eps}\left[\log(1/\eps)\left(\frac{M_\eps^\beta}{D_\eps^\beta}\right)^2\right]. \]
	Hence using Proposition \ref{lemma::conv_quotient} and Lemma \ref{lemma::min_particle}, together with the fact that \\ $\sqrt{\log(1/\eps)}M_\eps(\Os)$ converges to $\sqrt{\pi/2}\mu'(\Os)$ (which is positive almost surely) and that $\mathbb{E}[D_\eps^\beta]=\int_x Z_\eps^\beta(x)$ is bounded, we can conclude by letting $\beta\to \infty$, then $\eta\to 0$, and finally $\eps\to 0$. 
\end{proofof}

\section{$\star$-scale invariant kernels}\label{sec::starscale}

In this section we prove Theorem \ref{theorem::convergencestar} using a simple adaptation of our arguments from the previous section. Recalling the set-up, we have: 
\begin{itemize}
	\item $\theta$ a Radon measure with H\"{o}lder continuous density, and satisfying (\ref{eqn::cond_theta});
	\item $k:[0,\infty)\to \R$, a compactly supported and positive-definite $C^1$ function with $k(0)=1$; and 
	\item $h$ a $\star$-scale invariant field on $\R^d$ with covariance kernel 
	\[ K(x,y)=\int_1^\infty \frac{k(u|x-y|)}{u} \, du. \] 
\end{itemize}
We would like to prove that if $h_\eps(x)=h\star \theta_\eps(x)$ is the $\theta$-convolution approximation to $h$, the signed measures
\[D_\eps(dx) := (-h_\eps(x)+\sqrt{2d}\log(1/\eps))\e^{\sqrt{2d}h_\eps(x)}\eps^{d} \, dx\]
converge weakly in probability to a limiting measure. Moreover, we would like to show that this limiting measure is equal to the measure defined in \cite{DSRV,JS} (see Theorems \ref{theorem::critical_star} and \ref{theorem::critical_sh}).
\medbreak

\begin{proof}
	First pick $g$ such that $g\star g(u)=k(u)$ (we can do this by our assumptions on $k$, \cite{convroots}). Then we can define a field $h$ with the correct covariance structure by setting
	\begin{equation}
	\label{eqn::whitenoisedecomp}
	h(x):= \int_1^\infty \int_{\R^d} \frac{g(y-xu)}{\sqrt{u}} W(dy,du),
	\end{equation}
	where $W(\cdot, \cdot)$ is a standard space-time white noise. It is then proved in \cite{DSRV} that if we let 
	\[\tilde{h}_\eps(x):=\int_1^{\frac{1}{\eps}} \int_{\R^d} \frac{g(y-xu)}{\sqrt{u}} W(dy,du),\] the signed derivative measures 
	$\tilde{D}_\eps(dx) := (-\tilde h_\eps(x)+\sqrt{2d}\log(1/\eps))\e^{\sqrt{2d}\tilde{h}_\eps(x)}\eps^{d} \, dx$
	converge almost surely to a positive limiting measure $\mu'$. It is further shown in \cite{DSRV2} that \[\tilde{M}_\eps(dx):=\sqrt{\log(1/\eps)}\e^{\sqrt{2d}\tilde{h}_\eps(x)}\eps^d \,dx\] converges to $\sqrt{2/\pi}\mu'$ in probability and in \cite{JS} that 
	\[M_\eps:=\sqrt{\log(1/\eps)}\e^{\sqrt{2d}h_\eps(x)}\eps^d \,dx\] also converges to $\sqrt{2/\pi}\mu'$ in probability. 
	
	To prove the convergence of $D_\eps(dx)$ we use the same strategy as for the proof of Theorem \ref{proposition::convergence}, now letting $\tilde{h}_\eps$ play the role of the circle average. In particular we need only prove Proposition \ref{lemma::conv_quotient} (the result then following by Lemma \ref{rmk::min_particle} and Lemma \ref{lemma::min_particle} in exactly the same way.) We observe that:
	\begin{itemize}
		\item $\tilde{h}_\eps(x)$ is a (time-changed) Brownian motion for each $x\in \R^d$; and
		\item $\cov(h_\eps(x),\tilde{h}_\delta(x))=\log(1/(\eps\wedge \delta))+\text{O}(1)$, so we can define $\lambda_\eps(x)$, $Y_\eps(x)$ and $\rho^\eps_\delta(x)$ as in Lemmas \ref{lemma::compare_average} and \ref{rmk::covYyhx}, and the statements of these lemmas will still hold.
	\end{itemize}
	This is enough to prove (\ref{eqn::firstmoment}) and step 1 of (\ref{eqn::secondmoment}). For step 2 we need to explain how we define a few things. We let $r_\eps$ be chosen as before, and without loss of generality we assume that $\text{supp}(k)\subset B(0,1)$.  It is then easy to check using the definition of $h$ that $h_\eps(z)$ and $(\tilde{h}_{\eta}(x)-\tilde{h}_{r_\eps-\eps}(x))_{\eta\leq r_\eps-\eps}$ are independent for all $z\notin B(x,r_\eps)$. We let $\mathcal{F}_\eps=\sigma(\{\tilde{h}_u(x): u\geq r_{\eps}-\eps\})\vee \sigma(\{h_\eps(z):z\in D\setminus B(x,r_\eps)\})$ so that $(\tilde{h}_{\eta}(x)-\tilde{h}_{r_\eps-\eps}(x))_{\eta\leq r_\eps-\eps}(x)$ is independent of $\mathcal{F}_\eps$, and $M_\eps^{\beta,out}/D_\eps^{\beta,out}$ is measurable with respect to it. This is what we will use in place of $\F_{r_\eps-\eps}$ from the original proof. Using standard properties of Gaussian processes we see that we can also write $Y_\eps(x)=Y_\eps^1(x)+Y_\eps^2(x)$ where $Y_\eps^1(x)$ is measurable with respect to $\mathcal{F}_\eps$ and $Y_\eps^2(x)$ is independent of it. From this point onwards we can define everything in the same way, and steps 3 and 4 follow, using only properties of the 3d Bessel process. 
	
	To conclude, we need only complete step 5, since step 6 is a straightforward consequence of this (as in the original proof.) For this step we note that by our assumption on $\text{supp}(k)$,   $(\tilde{h}_{\delta+\eta}(x)-\tilde{h}_\delta(x))_{\eta\geq 0}$ and $(\tilde{h}_{\delta+\eta}(y)-\tilde{h}_\delta(y))_{\eta\geq 0}$ are independent as soon as $|y-x|\geq \delta$. Since this is the only extra property we used in this step, the proof of Proposition \ref{lemma::conv_quotient} goes through.
	
\end{proof}

\begin{remark}
	We remark here that the authors in \cite{DSRV,DSRV2} suggest that their constructions should hold for more general kernels than the $\star$-scale invariant ones. In particular, for any positive definite kernel of the form 
	$K(x,y)=-\log(|x-y|)+g(|x-y|)$
	with $g$ continuous, one has a white-noise decomposition for the corresponding field $h$, analogous to (\ref{eqn::whitenoisedecomp}). This means that the theory in \cite[Appendix D]{DSRV2} should go through, and as a consequence, the result of Theorem \ref{theorem::convergencestar} should also hold. More generally we conjecture that Theorem \ref{theorem::convergencestar} should hold for any $K$ satisfying (\ref{eqn::kernelform}). 
\end{remark}



\providecommand{\bysame}{\leavevmode\hbox to3em{\hrulefill}\thinspace}
\providecommand{\MR}{\relax\ifhmode\unskip\space\fi MR }
\providecommand{\MRhref}[2]{%
	\href{http://www.ams.org/mathscinet-getitem?mr=#1}{#2}
}
\providecommand{\href}[2]{#2}




\end{document}